\documentclass[10pt]{amsart} 

\usepackage{amsmath,amstext,amscd, amssymb, txfonts, color}
\usepackage{pinlabel}
\usepackage[normalem]{ulem}
\usepackage[colorlinks=true, pdfstartview=FitV, linkcolor=blue, citecolor=blue, urlcolor=blue]{hyperref}
\usepackage[abs]{overpic}
\usepackage{xcolor,varwidth}
\usepackage{enumerate}

\makeatletter
\def\@tocline#1#2#3#4#5#6#7{\relax
	\ifnum #1>\c@tocdepth 
	\else 
	\par \addpenalty\@secpenalty\addvspace{#2}%
	\begingroup \hyphenpenalty\@M
	\@ifempty{#4}{%
		\@tempdima\csname r@tocindent\number#1\endcsname\relax
	}{%
		\@tempdima#4\relax
	}%
	\parindent\z@ \leftskip#3\relax \advance\leftskip\@tempdima\relax
	\rightskip\@pnumwidth plus4em \parfillskip-\@pnumwidth
	#5\leavevmode\hskip-\@tempdima #6\nobreak\relax
	\ifnum#1<0\hfill\else\dotfill\fi\hbox to\@pnumwidth{\@tocpagenum{#7}}\par
	\nobreak
	\endgroup
	\fi}
\makeatother
\let\oldtocsection=\tocsection
\let\oldtocsubsection=\tocsubsection
\let\oldtocsubsubsection=\tocsubsubsection
\renewcommand{\tocsection}[2]{\hspace{0em}\oldtocsection{#1}{#2}}
\renewcommand{\tocsubsection}[2]{\hspace{1em}\oldtocsubsection{#1}{#2}}
\renewcommand{\tocsubsubsection}[2]{\hspace{2em}\oldtocsubsubsection{#1}{#2}}

\newcommand{\tc}[2]{\textcolor{#1}{#2}}
\definecolor{cerulean}{rgb}{0,.48,.65} 
\definecolor{magenta}{rgb}{.5,0,.5} 
\definecolor{dred}{rgb}{.5,0,0} 
\definecolor{green}{rgb}{0,.5,0} 
\definecolor{blue}{rgb}{0,0,1} \newcommand{\blue}[1]{\tc{blue}{#1}}
\definecolor{black}{rgb}{0,0,0} 
\definecolor{dgreen}{rgb}{0,.3,0} 
\definecolor{vdred}{rgb}{.3,0,0} 
\definecolor{red}{rgb}{1,0,0} \newcommand{\red}[1]{\tc{red}{#1}}
\definecolor{salmon}{rgb}{0.98,0.50,0.45} 
\definecolor{gray}{rgb}{.5,.5,.5} 
\definecolor{seagreen}{rgb}{0.13,0.70,0.67} 
\definecolor{chartreuse}{rgb}{0.40,0.80,0.00}
\definecolor{cornflower}{rgb}{0.39,0.58,0.93} 
\definecolor{gold}{rgb}{0.80,0.68,0.00}

\theoremstyle{plain}

\newtheorem{theorem}{Theorem}
\newtheorem{thm}{Theorem}[section]
\newtheorem{lemma}[thm]{Lemma}

\newtheorem{cor}[thm]{Corollary}

\newtheorem{prop}[thm]{Proposition}
\newtheorem{proposition}[thm]{Proposition}

\newtheorem*{conjproblem*}{The conjugacy problem for $H$}
\newtheorem*{0twistedproblem*}{The 0-twisted-conjugacy problem}
\newtheorem*{Htwistedproblem*}{The $\H$-twisted conjugacy problem}
\newtheorem*{Itwistedproblem*}{The I-twisted conjugacy problem}

\theoremstyle{definition}

\newtheorem{remark}[thm]{Remark}

\newtheorem{Example}[thm]{Example}

\newtheorem{Open questions}[thm]{Open questions}
\newtheorem{Open question}[thm]{Open question}
\newtheorem{Open problems}[thm]{Open problems}
\newtheorem{Open problem}[thm]{Open problem}

\def\cal{\mathcal}

\def\Bbb{\mathbb}
\def\bar{\overline}
\def\H{\mathcal{H}}
 
\def\Z{\Bbb{Z}}
\def\N{\Bbb{N}}

\def\ni{\noindent}

\def\rank{\hbox{\rm rank}}

\def\wt{\hbox{\rm wt}}

\def\CL{\hbox{\rm CL}}

\def\F+L{\hbox{$\textup{F}\!_+\textup{L}$}}

\def\ssm{\smallsetminus}

\def\Aut{\hbox{\rm Aut}}

\def\onto{{\kern3pt\to\kern-8pt\to\kern3pt}}

\def\<{\langle}
\def\>{\rangle}
\def\|{{\ |\ }}

\def\a{\alpha}
\def\b{\beta}

\def\G{\Gamma}

\newcommand{\set}[1]{\left\{#1\right\}}

\newcommand{\abs}[1]{\left|#1\right|}
\renewcommand{\ni}{\noindent}

\def\*{^{\star}}

\newcommand{\pieces}[1]{{|\!|{#1}|\!|_\pi}} 
\newcommand{\case}[1]{\medskip
	\noindent{\bf Case #1.} \nopagebreak}
\newcommand{\step}[1]{\medskip
	\noindent{\bf Step #1.} \nopagebreak}

\newcommand{\commentA}[1]{\marginpar{\tiny\begin{center}\textcolor{blue}{#1}\end{center}}}

\begin{document}

	\title{Conjugacy in a family of free-by-cyclic groups}
	
	\author{M.\ R.\ Bridson, T.\ R.\ Riley and A.\ W.\ Sale}
	
	\date{22 May 2025}

	\begin{abstract}
		\ni  We analyse the geometry and complexity 
		of the conjugacy problem in a family of free-by-cyclic groups 
		$H_m=F_m\rtimes\Z$ 
		where the defining free-group automorphism is positive and polynomially growing.  
		We prove that the conjugator length function of $H_m$ is linear, and
		describe  polynomial-time solutions to the conjugacy problem and conjugacy search problem in $H_m$. 
		
		\smallskip
		\ni \footnotesize{\textbf{2020 Mathematics Subject Classification:  20F65, 20F10}}  \\ 
		\ni \footnotesize{\emph{Key words and phrases:} free-by-cyclic groups, conjugacy problem,  conjugator length}
	\end{abstract}
	
	\thanks{We gratefully acknowledge the financial support of the Royal Society (MRB), the Simons Foundation (TRR--Simons Collaboration Grant 318301), and the National Science Foundation (TRR--GCR2428489 and OIA-2428489). }

	\maketitle

\section{introduction}
 Suppose $G$ is a finitely generated group. The \emph{conjugacy problem}  asks  for an algorithm that, given any words $u$ and $v$ on the generators and their inverses, decides whether or not these words
  represent conjugate elements in $G$. We write $u\!\sim\!v$ to denote
 conjugacy.  The  \emph{conjugacy search problem} asks   for an algorithm that, given a pair of words $u$ and $v$ such that $u \!  \sim \! v$, will output a word $w$ with $uw=wv$ in $G$.  The \emph{conjugator length function} $\CL : \N \to \N$ quantifies these problems: $\CL(n)$ is the least integer  $N$ such that for all words $u$ and $v$ that  represent conjugate elements in $G$  and have length $|u|+|v|\le n$,  there is a word  $w$ of length at most $N$  such that  $uw=wv$ in $G$.  
The conjugator length functions of $G$ with respect to different finite generating sets are $\simeq$-equivalent, where $\simeq $ is the equivalence relation that identifies  
functions $\mathbb{N} \to \mathbb{N}$ that dominate each other modulo affine distortions of  their domain and their range. 
Extensive background on conjugator length can be found in \cite{BrRiSa}.

 Fix an integer $m \geq 1$ and let  $F = F(a_1, \ldots, a_m)$ be a rank-$m$  free group. 
	Define  $\varphi \in \Aut(F)$   by 
	$\varphi(a_i) = a_ia_{i-1}$  for $2\leq i \leq m$  and $\varphi(a_1)=a_1$.
	This paper concerns the  free-by-cyclic groups  $$ H_m  \ = \ F \rtimes_\varphi \Z \ = \  \langle a_1, \ldots, a_m , s \mid s^{-1}a_i s = \varphi(a_i) \rangle.$$
The   inclusions
$H_{m-1}\hookrightarrow H_m$ (excluding $a_m)$ and retractions $H_m\to H_{m-1}$ 
(killing $a_1$) will facilitate induction arguments.

The groups $H_m$ have many useful properties and have appeared regularly 
in the literature. They appear as
`hydra groups' in \cite{BaR1, DR, DER, Pueschel}. Each is the fundamental
group of a compact non-positively curved 2-complex built from squares
\cite{Samuelson};
in particular it is biautomatic and CAT$(0)$. 
Each can be expressed as a 2-generator
1-relator group, or as a free-by-cyclic
group $F_r\rtimes\Z$ with $r$ arbitrarily large \cite{Button07}. 
$H_2$ is famous as a
3-manifold group that is not subgroup separable \cite{BKS}.
But, most obviously, these groups $H_m$ serve as natural prototypes for the
mapping tori of free-group automorphisms that have maximal polynomial growth
\cite{BBMS, Bridson2, CM, Gersten12, Macura2, Macura4, Samuelson}, and 
for the most part this is how we shall regard them.

Our main results here are:

\begin{theorem} \label{CL of free-by-cyclic}
For all $m \geq 2$, the conjugator length function of $H_m$ satisfies $\CL(n) \simeq n$.
\end{theorem}

\begin{theorem} \label{CP of free-by-cyclic}
For all $m\geq 2$, there exist algorithms solving the  conjugacy problem and the conjugacy search problem of $H_m$ in time polynomial in the sum of the lengths of the input words. 
\end{theorem}

Our proofs of these theorems are intertwined and are constructive. We first describe in detail an
algorithmic procedure that solves the  conjugacy problem and the conjugacy search problem of $H_m$; this is summarized in Section~\ref{the alg}. A na\"ive analysis shows that
this algorithm will output a conjugator whose length is bounded by a  quadratic function
of the lengths of the input words, but a more careful analysis shows that 
with minor modifications this quadratic bound can be reduced to 
a linear one---see Remark~\ref{upgrade to linear remark}. 

We regard these results as a significant step towards
bounding the complexity of the conjugacy problem and conjugacy search problem
in arbitrary free-by-cyclic groups (where the free group has finite rank, which will be a standing assumption throughout our discussion). Free-by-cyclic groups  provide a rich and challenging arena for the study of geometric invariants 
of groups associated with various weak forms of non-positive
curvature (as discussed in \cite{BridsonGroves}, for example).  For any free-by-cyclic group, there is an algorithm solving its word problem in polynomial time \cite{Schleimer}.
There are also  algorithms solving the conjugacy problem \cite{BMMV, BridsonGroves}, but  these do not provide reasonable
bounds on time complexity.  In particular, it is unknown
whether the conjugacy problem and conjugacy search problem  can be solved
in polynomial time. The results in this paper add weight to the conviction that this
is likely.

When a free-by-cyclic group is hyperbolic or its conjugacy problem and conjugacy search problem  can be solved in linear time.  These  are basic examples of a much more general result:   there are polynomial-time solutions for all  groups which are hyperbolic relative to a  finite family of peripheral subgroups in which one can solve the corresponding   problems in polynomial time---see   \cite{Bumagin2, EH, JOR, OConnor}.  Free-by-cyclic groups are hyperbolic relative to a finite family of free-by-cyclic subgroups, each of which has the property that the defining automorphism is  polynomial-growing---see \cite{BFWOld} for history and references.  Thus  the search for a
polynomial time solution to the conjugacy problem reduces to the case where
the defining automorphism is  polynomial-growing, and
Theorem~\ref{CP of free-by-cyclic} solves this problem for a natural
family of prototypes.

It seems reasonable to expect that the conjugator length
function of an arbitrary free-by-cyclic group is linear. This is true
in the hyperbolic case \cite{BrH, Lysenok},
but beyond that little is known. However, by appealing to the relative hyperbolicity result mentioned above, one can again reduce to the case where the defining automorphism
is polynomially growing, because Sale proved \cite{AS}
that if $G$ is non-degenerately hyperbolic relative to parabolic subgroups $P_\omega\
(\omega \in \Omega)$, then $\CL_G(n) \simeq \max\{\CL_{P_\omega}(n)  : \omega\in\Omega\}+n$.  As in the case of complexity, Theorem~\ref{CL of free-by-cyclic}
assures us that the desired bound $\CL_G(n)\simeq n$ is valid in a natural class of prototypes.  

The role that the groups $H_m$ play as prototypes among free-by-cyclic groups is analogous to the role that 
the   \emph{model filiform groups}
   $$ \G_m= \Z^m \rtimes \Z    \ = \  \langle \, a_1, \ldots, a_m , s \mid  a_ia_j=a_ja_i \  \forall i,j, \ s^{-1}a_i s =  a_i a_{i-1} \forall i \geq 2, \ s^{-1}a_1 s =  a_1  \, \rangle,$$
   play among (free-abelian)-by-cyclic groups.    In \cite{BrRi2} we prove that, in contrast to Theorem~\ref{CL of free-by-cyclic},
    the conjugator length function of $\G_m$ is polynomial of degree $m$.

The proofs is this paper are largely combinatorial and 
typically require a delicate analysis of cases. We make heavy use of the notion of `decomposing reduced words into  pieces.'  
This tool is from \cite{DR}, and can be viewed as a special case of the train-track machinery of \cite{BeHa, BFHI, BFH}. 
We have favoured using pieces here because they lend themselves
well to the detailed study of cancellation in the free group that we need, 
and to the precise understanding of how words in the free group grow under
iteration of the automorphism.
Nevertheless, we have structured our proofs with an eye to how they might be 
adapted to cover more general polynomially growing automorphisms. In particular,
we have not relied on any of the alternative ways of viewing $H_m=F\rtimes\Z$
that were discussed earlier. Instead, we consistently view $H_m$ as a semidirect
product and work with elements in the form $wt^n$, where $w\in F$ and $n\in\Z$.
From this viewpoint, the complexity of the conjugacy problem in $H_m$
translates into a collection of {\em twisted conjugacy problems} in $F$.
A benefit of this direct approach is that the outlines of various arguments 
carry over to the general case.

In a sequel to this paper \cite{BrRiSa3}, we will present a different approach to
the conjugacy problem in $H_m$ that does rely
on one of these alternative perspectives,
namely the fact that $H_m$ can be obtained
from $\Z^2$ by a sequence of HNN extensions with cyclic amalgamated groups.
The more geometric  arguments in \cite{BrRiSa3} are  
framed with an eye to further generalisations.

In the next section we will translate the conjugacy problem in $H_m$ 
into a suite of twisted conjugacy
problems in $F$ and lay out the framework for the rest of this article.   
It is the analysis of these twisted problems that forms the bulk of what follows.
Throughout, we shall write $H$ in place of $H_m$ when there is no
danger of ambiguity.

\section{Reduction to twisted problems in $F$} \label{reductions}

Conjugacy in the free group $F = F(a_1, \ldots, a_m)$ can be fully understood thanks to the following well-known result (e.g.\ \cite{LS}).

\begin{lemma} \label{CP in F lemma}
If words $u$ and $v$ on $a_1^{\pm 1}, \ldots, a_m^{\pm 1}$ represent conjugate elements of $F$, then there is a word $w$ on $a_1^{\pm 1}, \ldots, a_m^{\pm 1}$ which is a concatenation of a prefix of $u^{-1}$ with a suffix of $v$ such that $uw=wv$ in $F$.  (If $v$ is cyclically reduced---that is, $vv$ is reduced---then $w$ need only be  a prefix of $u^{-1}$.)  

Assume $u$ does not represent the identity.  Take $k$  to be the   maximal integer such  that there exists $u_0$  with $u_0^k = u$ in $F$.  (So $u_0$ generates the centralizer of $u$ in $F$.)  Then for any such $w$ and $u_0$, $$\set{ \left. W \in F \phantom{u_0^l} \hspace*{-3mm}  \ \right|   \   uW=Wv \text{ in } F }  \ = \       \set{ \left. u_0^l  \ \right| l \  \in \Z }w.$$  
\end{lemma}

 \begin{conjproblem*} 
Given words $u$ and $v$ on  $a_1^{\pm 1}, \ldots, a_m^{\pm 1}, s^{\pm 1}$ does  there exist a word $w$ on  $a_1^{\pm 1}, \ldots, a_m^{\pm 1}, s^{\pm 1}$ such that  $uw    =   wv$ in $H$?
\end{conjproblem*}

We will use a standard free-by-cyclic {\em{\bf{normal form}}}: each element in $H
=F\rtimes \<s\>$ 
can be expressed uniquely as  $\tilde{u}s^p$  for some reduced word  $\tilde{u}$ on $a_1^{\pm 1}$, \ldots, $a_m^{\pm 1}$  and  $p \in \Z$.

Suppose we have conjugate elements $u$ and $v$  of $H$ expressed in normal form as $\tilde{u}s^p$ and $ \tilde{v} s^q$, respectively.  The conjugacy relation $uw = wv$ in $H$, where $w$ has normal form $\tilde{w}s^r$, implies $p=q$ and amounts to the 
`$\varphi$-twisted conjugacy relation' 
	\begin{equation} \label{twisted conjugacy}
	\tilde{u} \varphi^{-p}(\tilde{w})  \ = \  \tilde{w} \varphi^{-r}(\tilde{v}) \qquad 
	\text{ in the free group }  F.
	\end{equation}
This problem is much harder than the conjugacy problem for free groups, although as we will see, in some instances its solution ultimately reduces to Lemma~\ref{CP in F lemma}.

In the instance where $p=q = 0$, the conjugacy problem in $H$ therefore amounts to:

\begin{0twistedproblem*} 
Given words $\tilde{u}$, $\tilde{v}$ on $a_1^{\pm 1} , \ldots , a_m^{\pm 1}$, do  there exist $r \in \Z$ and $\tilde{w} \in F$ such that  $\tilde{u}  \tilde{w}     =    \tilde{w} \varphi^{-r}(\tilde{v})$ in $F$?
\end{0twistedproblem*} 

 	This problem is addressed in Section~\ref{0-twisted section}. Proposition~\ref{0 twisted conj problem} gives both complexity and conjugator length bounds.

The conjugacy problem in $H$ with $p<0$ is equivalent to that with $p>0$ since we can exchange  $u$ and $v$ with their inverses.  So  in place of $p \neq 0$, let us just  consider   $p=q >0$.  
	
From $uw = wv$ in $H$, we get that  $u(u^jw) = (u^jw) v$ for all $j \in \Z$, and so   there exists a $w$ such that $uw = wv$ in $H$ \emph{and} such that the normal form of $w$ is $\tilde{w}s^r$ for some reduced word $\tilde{w}$  on $a_1^{\pm 1} , \ldots , a_m^{\pm 1}$ and some integer $r$ satisfying  $0\leq r <p$.	
	
	The Cayley graph of $F$ is a tree, so the geodesics joining $1$, $\tilde{u}$, $\tilde{w}$, and $\tilde{u}\varphi^{-p}(\tilde{w})$ form either the `$\H$-configuration' (left) or the `I-configuration' (right)  shown in   Figure~\ref{fig:fbc determine X}.  Accordingly, we can find prefixes $u_0,v_0$ and suffixes $u_1,v_1$ of $\tilde{u}, \varphi^{-r}(\tilde{v})$, respectively,
	and two words $x, y$,  at least one   the empty word,  
	 such that $\tilde{w} = u_0 x v_0^{-1}$ and $\varphi^{-p}(\tilde{w}) = u_1^{-1} x v_1$, 
	where $\tilde{u} = u_0 y u_1$ and $\varphi^{-r}(\tilde{v}) = v_0 y v_1$ as freely reduced words.

	\begin{figure}[ht]
\begin{overpic}[
scale=1.0,unit=1mm]{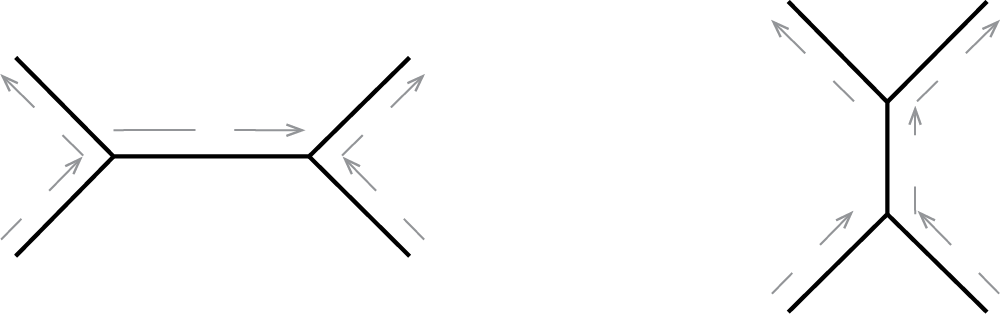}
 \put(-1,21){\small{$\tilde{u}$}}     
 \put(35,3){\small{$\tilde{w}$}}     
  \put(-0.5,3){\small{$1$}}     
 \put(35.5,21.5){\small{$\tilde{u}\varphi^{-p}(\tilde{w})=\tilde{w}\varphi^{-r}(\tilde{v})$}}     
 \put(2,9){\small{$u_0$}}     
 \put(2.5,15.5){\small{$u_1$}}     
 \put(31,15.5){\small{$v_1$}}     
 \put(31.9,9.1){\small{$v_0$}}     
 \put(67.3,4.3){\small{$u_0$}}     
 \put(80.3,4.2){\small{$v_0$}}     
 \put(67.4,20.2){\small{$u_1$}}     
 \put(80,20.5){\small{$v_1$}}     
\put(84,-2){\small{$\tilde{w}$}}     
 \put(64,26){\small{$\tilde{u}$}}     
  \put(65,-2){\small{$1$}}     
 \put(84,26){\small{$\tilde{u}\varphi^{-p}(\tilde{w})=\tilde{w}\varphi^{-r}(\tilde{v})$}}     
 \put(17,14.5){\small{$x$}}     
 \put(76.3,12){\small{$y$}}    
\end{overpic}\hspace{15mm}
\caption{The two possibilities for the relative locations of $1$, $\tilde{u}$, $\tilde{w}$, and $\tilde{u}\varphi^{-p}(\tilde{w})$ in the Cayley graph of $F$: the `$\H$-configuration' on the left and the `I-configuration' on the right.} 
  \label{fig:fbc determine X}
\end{figure}
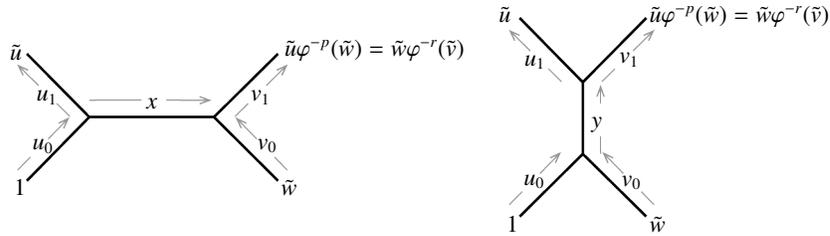

In the $\H$-configuration, the  conjugacy problem amounts to:

\begin{Htwistedproblem*} 
Given  reduced words $\tilde{u}$, $\tilde{v}$ on $a_1^{\pm 1}, \ldots , a_m^{\pm 1}$  and $p >0$,
do there exist $0 \leq r < p$ and   words $x , u_0,v_0,u_1,v_1 \in F$ such that $\tilde{u} = u_0u_1$ and $\varphi^{-r} (\tilde{v}) = v_0v_1$, as words, and 
		\begin{equation*}   
	    \varphi^{-p}(u_0  x  v_0^{-1})  \ = \  u_1^{-1} x v_1  \  \text{  in }  \   F? 
	  	\end{equation*}
\end{Htwistedproblem*}

	Most of the difficulties and technicalities lie in this problem.
	Section~\ref{prefixes section} addresses a special case of the $\H$-twisted conjugacy problem when $\tilde{u}$ is the empty word. 	As explained there, this amounts to understanding  the structure of common prefixes of a word $\tilde{w}$ with its image $\varphi^r(\tilde{w})$, a crucial ingredient in the general case of the $\H$-twisted conjugacy problem.
	In Section~\ref{H-twisted section1}, specifically Proposition~\ref{conjugator-like problem in F}, we describe how to use a solution (that is, the integer $r$ and the words $x,u_0,u_1,v_0,v_1$) to obtain a `nicer' solution in which $x$ is replaced by a word $X$ whose structure can be described in terms of short chunks that come as subwords of $u_0,u_1,v_0,v_1$, their inverses and the iterates under powers of $\varphi$.
This `chunky' structure enables us, in Section~\ref{H-twisted section2}, to find a short conjugator and describe a polynomial-time solution to this problem.

 	Unfortunately, the short conjugator that one obtains from the `chunky' structure of Proposition~\ref{conjugator-like problem in F} actually only has a quadratic upper bound in terms of its  length in $H$, relative to those of $u$ and $v$.
	Lemma~\ref{lem:swap for a linear conjugator} describes how to replace one of the chunks with a suitable power of $s$ to obtain a linearly bounded conjugator.

In the I-configuration, $x$ is the empty word and $\tilde{w} = u_0   v_0^{-1}$, and so the  conjugacy problem amounts to:

\begin{Itwistedproblem*} 
Given  reduced words $\tilde{u}$, $\tilde{v}$ on $a_1^{\pm 1} , \ldots , a_m^{\pm 1}$ and $p >0$,
do  there exist $0 \leq r < p$ and   prefixes $u_0$ of $\tilde{u}$ and $v_0$ of $\varphi^{-r}(\tilde{v})$   such that 
		\begin{equation*}   
	    \tilde{u} \varphi^{-p}(u_0    v_0^{-1})  \ = \  u_0   v_0^{-1} \varphi^{-r}(\tilde{v})  \  \text{  in }  \   F? 
	  	\end{equation*}
\end{Itwistedproblem*}	

This problem is easy to solve by an exhaustive search. 
Indeed, given $\tilde{u},\tilde{v}$ and $p$ as in the I-twisted conjugacy problem, define $\cal{I}$ to be the set of all pairs $(\tilde{w},r)$, where $0\le r  < p$, and $\tilde{w}$ is a word of the form $UV$ where $U$ is a prefix of $\tilde{u}$ and $V^{-1}$ is a prefix of $\varphi^{-r}(\tilde{v})$.
A solution, if it exists, can be found by applying the solution to the word problem in $F$ to check the validity of each equation $\tilde{u} \varphi^{-p}(\tilde{w}) = \tilde{w} \varphi^{-r}(\tilde{v})$ for each   $(\tilde{w},r) \in \cal{I}$.

Like the conjugacy problem, the 0-twisted-conjugacy problem, the $\H$-twisted conjugacy problem, and the  I-twisted conjugacy problem all have `search' variants in which one is given that a collection of integers and words solving the problem exists and is required to exhibit one.

For $g \in H$, let $\abs{g}_H$ denote the length of a shortest word   on $\{a_1,\ldots,a_m,s\}$ that represents $g$.
If $g \in F$, let $\abs{g}_F$ be the  length of a shortest word   on  $\{a_1,\ldots, a_m\}$ that represents $g$.  For a word 
(not necessarily reduced) $w$, $\ell(w)$ denotes the number of letters in $w$.     

The following summarises results from Proposition~\ref{0 twisted conj problem} and Corollary~\ref{conjugator-like problem in F solution}, along with the discussion above for the I-twisted case. 

 \begin{prop} \label{twisted search} With the notation established above,
in $H_m$,
the 0-twisted conjugacy problem, the $\H$-twisted conjugacy problem, and the  I-twisted conjugacy problem can each be solved by deterministic algorithms
with input $(\tilde{u}, \tilde{v}, p)$
whose running time is bounded by a polynomial in $p+ \ell(\tilde{u}) + \ell(\tilde{v})$.  And the same is true for the `search' variants of these problems.
\end{prop}

 Given that $\abs{p} \leq \ell(u)$,   $\ell(\tilde{u})  \leq C  \ell(u)^m$, and $\ell(\tilde{v})  \leq C  \ell(v)^m$ for a suitable constant $C>0$, Theorem~\ref{CP of free-by-cyclic} follows from Proposition~\ref{twisted search} as per the above discussion.   
 
Turning to  conjugator length, since we want to bound conjugator length in $H$,   bounds pertaining to the three subordinate problems in $F$ will need to be given in terms of the word metric on $H$ rather than on $F$.    The issue of comparing these two metrics is delicate and is the subject of Section~\ref{distortion section}.

In summary, the article is structured as follows.
Section~\ref{defs and conventions section} introduces \emph{piece decompositions}, 
a useful technical tool.  
Section~\ref{distortion section} addresses the distortion of $F$ in $H$.
Section~\ref{0-twisted section} deals with the 0-twisted conjugacy problem.
After   technical results  in Section~\ref{prefixes section} on common prefixes of words $\tilde{w} \in F$ and their iterates  under powers of $\varphi$, we   handle the $\H$-twisted-conjugacy problem instances in Sections~\ref{H-twisted section1} and \ref{H-twisted section2}.  
The conjugator length argument in $H$ is completed in Section~\ref{sec:CL of H}. Section~\ref{the alg} summarizes how to assemble our results into an algorithm for Theorem~\ref{CP of free-by-cyclic}. 
In Section \ref{s:last} we focus on the structure of $H_m$ as an
iterated HNN extension and outline an alternative solution to the conjugacy problem.

 \section{Preliminaries: piece decompositions, definitions, and conventions} \label{defs and conventions section}

\subsection{Some notations, and conventions}\label{subsec:defs and conventions}

For a word $w$ on a set of letters, we let $\ell(w)$ denote its length.
As mentioned in Section~\ref{reductions}, $\abs{\cdot}_H$ and $\abs{\cdot}_F$ denote the word length of an element in $H$ or $F$ respectively, with respect to generating sets $\{a_1,\cdots , a_m , s \}$ and $\{a_1,\cdots , a_m \}$ respectively. 

The \emph{rank}, written $\rank(w)$, of a word $w$ on $a_1^{\pm 1}, \ldots, a_m^{\pm 1}$ is the maximal $i$ such that $a_i^{\pm 1}$ appears in $w$.  The empty word has rank $0$. The rank of  $g \in F$  is the rank of the reduced word $w$ representing $g$.

For a word $w$, when we write $\varphi(w)$ we mean the reduced word representing $\varphi(w)$ in $F$.

 \subsection{Positivity of $\varphi$ and $\varphi^{-1}$}\label{subsec:positive inverse}

 The inverse of $\varphi$ is
\begin{equation}\label{eq:phi inverse}
\varphi^{-1}(a_i)  \ = \  
\begin{cases}
a_{2k}a_{2(k-1)}\cdots a_2 a_1^{-1} a_3^{-1}\cdots a_{2k-1}^{-1} & \textrm{ when $i=2k$,}\\ \\
a_{2k+1}a_{2k-1}\cdots a_1 a_2^{-1} a_4^{-1} \cdots a_{2k}^{-1} & \textrm{ when $i=2k+1$.}
\end{cases}
\end{equation}

A useful feature of $\varphi$ is that it is a \emph{positive automorphism}: whenever $g \in F$ is represented by a   positive word,   so is $\varphi(g)$.  This is not true of $\varphi^{-1}$.  However $\varphi^{-1}$ is positive with respect to the basis $b_1, \dots, b_m$, defined by $b_i = a_i^{(-1)^{i+1}}$ for all $i$, since 
\begin{equation}\label{eq:phi inverse positive}
\varphi^{-1}(b_i)  \ = \  
\begin{cases}
b_{2k-1} \cdots b_3   b_1   b_2  \cdots b_{2(k-1)}  b_{2k}  & \textrm{ when $i=2k$,}\\ \\  
b_{2k+1}b_{2k-1}\cdots b_1 b_2  b_4  \cdots b_{2k}  & \textrm{ when $i=2k+1$.}
\end{cases}
\end{equation}

\subsection{Pieces and their types}\label{subsec:pieces}

 We will find it useful to split $w$ into \emph{pieces} that behave well
when one takes iterated images under $\varphi$.  
A \emph{rank-$i$ piece} in $w$ is a maximal subword of one of the following 
four \emph{types}:
		$$a_iu,\ \ ua_i^{-1}, \ \ a_iua_i^{-1}, \ \ u,$$
where $u$ is a (possibly empty) word of rank at most $i-1$. Pieces of the first three types are said to be of {\em strict rank-$i$}.
Each  rank-$i$ word   can be expressed as a concatenation of a minimal number of rank-$i$ pieces in a unique manner.  We call this \emph{the rank-$i$ decomposition of $w$} and refer to the pieces involved as the \emph{pieces of $w$}.  We denote the number of these pieces by $\pieces{w}$. For example, $w = a_3 a_2 a_1^{-1} a_3 a_3 a_1 a_3^{-1} a_2 = (a_3 a_2 a_1^{-1}) (a_3) (a_3 a_1 a_3^{-1}) (a_2)$ is a rank-3 word with     $\pieces{w} =4$.

For $g \in F$ we write $\pieces{g} := \pieces{w}$, where $w$ is a reduced word representing $g$.

We need the following facts about pieces:

 \begin{lemma}\label{lem:pieces exercise0}
	If  a reduced word $\pi$ is a piece of rank $i$, then both $\varphi(\pi)$ and $\varphi^{-1}(\pi)$ are also pieces of rank-$i$ and have the same type as $\pi$.
\end{lemma}

\begin{lemma}\label{lem:pieces exercise}
	Let $w$ be a reduced word of  rank $i$.  Let $w = \pi_1 \cdots \pi_p$ be its rank-$i$ decomposition. Then  there is no cancellation between pieces on applying $\varphi$ or $\varphi^{-1}$---that is,  for   $k=1, \ldots, p-1$,  the words $\varphi^{\pm 1}(\pi_{k+1})$ and $\varphi^{\pm 1}(\pi_k)$ start and end (respectively) with letters that are not mutual inverses.  
	As a consequence,  $\varphi^r(\pi_1) \cdots \varphi^r(\pi_p)$ is freely reduced and is the rank-$i$ decomposition of $\varphi^r(w)$ for all $r \in \Z$.
\end{lemma}

We leave the proofs of Lemmas~\ref{lem:pieces exercise0} and \ref{lem:pieces exercise} as   exercises.  Very similar observations are made in \cite{DR}.

\section{Growth rates and distortion} \label{distortion section}

 In order to analyze the $\varphi$-twisted conjugacy problem in $F$ we  will examine in Section~\ref{lem:fixed pieces} how free group elements grow in length on repeated application of $\varphi$.  Then in Section~\ref{subsec:tools for distortion} we 
establish some useful inequalities relating the  normal form of 
$g\in H$ to $|g|_H$.   
  
\subsection{Growth rates}

Our next   few results   lead into Proposition~\ref{prop:lower bound distortion}, which gives a precise estimate of how words  grow on repeated applications of $\varphi^{\pm 1}$.  (Cf.\ \cite{Levitt} in which  bounds are given, but with the constants depending on the group element.)

\begin{lemma}\label{lem:fixed pieces}
${\rm{Fix}}(\varphi) = \< a_1, \, a_2a_1a_2^{-1}\>$.
\end{lemma}

\begin{proof}
 That ${\rm{Fix}}(\varphi) \supseteq \< a_1, \, a_2a_1a_2^{-1}\>$ is straight-forward.  For the reverse inclusion, first observe that by Lemmas~\ref{lem:pieces exercise0} and \ref{lem:pieces exercise}, $w$ is fixed by $\varphi$ if and only if its pieces are all  fixed by $\varphi$. 
 So we can focus on the case where  $w$ is a single  non-empty piece $\pi = a_i^\delta u a_i^{-\delta'}$   such that  $\pi = \varphi(\pi)$ in $F$, and  $\delta,\delta' \in \{0,1\}$ are not both zero, and  $u$   a reduced word of rank at most $i-1$  with $i\geq 2$.     
Applying $\varphi$ to $\pi$ adds $\delta - \delta'$ to   
 the exponent sum of the $a_{i-1}$ present.  But  since $\pi = \varphi (\pi)$, the exponent sum of  the $a_{i-1}$  in  $\pi$ and $\varphi (\pi)$ must agree, and therefore    $\delta = \delta'=1$ and $\pi = a_i  u a_i^{-1}$ (and, in particular, $u$ is non-empty).  

It remains to show that $i=2$. Assume, for contradiction, that $i>2$.
	Well,  $\pi = \varphi (\pi)$ tells us that
	$a_i u a_i^{-1}=\varphi(a_i u a_i^{-1}) = a_i  a_{i-1}  \varphi(u) a_{i-1}^{-1}a_i^{-1}$, and so  $\varphi(u) = a_{i-1}^{-1} u a_{i-1}$.  
	Reapplying $\varphi$ multiple times gives
	$$\varphi^r(u)  \ = \  \varphi^{r-1}(a_{i-1}^{-1}) \cdots a_{i-1}^{-1} u a_{i-1} \cdots \varphi^{r-1}(a_{i-1}), \ \ \textrm{for $r\geq 1$.}$$
	By Lemma~\ref{lem:pieces exercise}, the number of pieces $p$ in the rank-$(i-1)$ decompositions of $u$ and $\varphi^r(u)$ are the same for all $r\geq 1$.
	Hence there is cancellation in our expression for $\varphi^r(u)$.  This cancellation can occur only at either end of $u$.
	For $r$ large enough, say $r>p+1$, we will need $u$ to completely cancel out under free reduction.
	In particular, there are some $\alpha$ and $\beta$ such that
	$$1 \ = \  \varphi^\alpha(a_{i-1}^{-1}) \ldots a_{i-1}^{-1} u a_{i-1} \ldots \varphi^{\beta}(a_{i-1})$$
	and so
	$$u  \ = \  a_{i-1} \cdots \varphi^{\alpha}(a_{i-1})\varphi^\beta(a_{i-1}^{-1}) \cdots a_{i-1}^{-1}.$$
	If $\alpha = \beta$, then we get complete cancellation and $u=1$, a contradiction.
	So we assume $\alpha \ne \beta$.
	We can count the number of pieces $p$ of $u$ by observing the locations of letters $a_{i-1}^{\pm 1}$ (we use here that $i>2$).
	We get $p=\pieces{u} = \alpha +\beta +1$ (the terms $\varphi^{\alpha}(a_{i-1})\varphi^\beta(a_{i-1}^{-1})$ merge into one piece).
	We also, from above, have
	$$\varphi^r(u) \  = \  \varphi^{r-1}(a_{i-1}^{-1}) \cdots \varphi^{\alpha+1}(a_{i-1}^{-1}) \varphi^{\beta+1}(a_{i-1}) \cdots \varphi^{r-1}(a_{i-1}),$$
	which has to cancel down to $p$ pieces.
	Since $i>2$, and $\alpha \ne \beta$, this cancels down to $2r-\alpha-\beta -3 = 2r- p - 2$ pieces. Since $r>p+1$, we get a contradiction. 
	 \end{proof}

The following argument is well known (cf.\ Lemmas~3.5 and 5.6 in \cite{BR1} and Example~3.3 \cite{Bridson2}).

\begin{lemma}\label{lem:growth of a_i}
	Let $i \geq 2$.
	There exist  $C_i,D_i >0$ such that for all $r \in \Z \ssm \set{0}$,
			$$C_i \abs{r}^{i-1}  \ \leq \  \abs{ \varphi^r (a_i) }_F \ \leq  \ D_i \abs{r}^{i-1}.$$
\end{lemma}

\begin{proof} 
	First we address the case $r>0$.
	Since $\varphi^r(a_i)$ is a positive word in this case, its length is the same as its length in the abelianisation of $F$.
	With respect to the basis $\{a_1,\dots,a_m\}$,
	the action of $\varphi$ on the abelianisation is via the matrix $\Phi$ 	with ones on the diagonal and immediately above, and zeros elsewhere.
	Direct calculation yields upper triangular matrices:
		\[
		\Phi  \ =  \ \begin{pmatrix}
		1 & 1 &  &  &   &   \\
		 & 1 & 1 &  &   &   \\
		  &  & \ddots & \ddots   & & \\
		  &  &  & \ddots  & \ddots   & \\
		  & &  &  & 1 & 1 \\
		 & &  &  &  & 1
		\end{pmatrix}, \  \qquad \  
	\begingroup
	\renewcommand*{\arraystretch}{1.5}
		\Phi^r \ = \  \begin{pmatrix}
		1 & {r \choose 1} & {r\choose 2} & \cdots & {r \choose m-1} \\
		 & 1 & {r \choose 1} & \cdots & {r \choose m-2} \\
		& &  \ddots  & & \vdots \\
		& & & 1  & {r \choose 1} \\
		 & & & & 1
		\end{pmatrix},
	\endgroup
		\]
where  ${r \choose j}$ is understood to be $0$ for $j >r$.

So	$$ \abs{\varphi^r(a_i)}_F \ = \  \sum_{j=0}^{i-1} {r \choose j},$$
which, as a function of $r\in\N$,
 is Lipschitz equivalent to ${r \choose {i-1}}\sim r^{i-1}$.

To deal with negative powers, we use the fact that $\varphi^{-1}$
is a positive automorphism with respect to the basis
$\{b_1,\dots,b_m\}$ described in Section~\ref{subsec:positive inverse}.
With respect to this basis, the action of $\varphi^{-1}$ on the abelianisation of
$F$ is given by
	\[
	\Psi \ = \  \begin{pmatrix}
	1 & 1 & 1 & \cdots & \cdots  & 1  \\
	& 1 & 1 & \cdots & \cdots  & 1  \\
	&  & \ddots &    & & \vdots \\
	&  &  & \ddots  &    & \vdots \\
	& &  &  & 1 & 1 \\
	& &  &  &  & 1
	\end{pmatrix}, \ \qquad \ 
	\begingroup
	\renewcommand*{\arraystretch}{1.5} 
	\Psi^r \ = \  
	\begin{pmatrix}
	1 & {r \choose 1} & {r+1\choose 2} & \cdots & {r+m-2 \choose m-1} \\
	 & 1 & {r \choose 1} & \cdots & {r+m-3 \choose m-2} \\
	& &  \ddots  & & \vdots \\
	& & & 1  & {r \choose 1} \\
	& & & & 1
	\end{pmatrix}.
	\endgroup
	\]
	So 	$$ \abs{\varphi^r(a_i)}_F \ = \  
	  \sum_{j=0}^{i-1} {r+j-1 \choose j} 
	  $$
	which, as a function of $r\in\N$, is Lipschitz equivalent to ${{r+i-2}
	 \choose {i-1}}\sim r^{i-1}$.
\end{proof}

A straight-forward induction (cf.\ \cite[Lemma 7.1]{DR}) gives:

\begin{lemma}\label{lem:fbc image of a_i}  For $i =2, \ldots, m$, 
\[\varphi^r(a_i) \ = \  \begin{cases} 
a_i \,  a_{i-1}  \varphi(a_{i-1}) \cdots \varphi^{r-1}(a_{i-1})  & \textrm{ when } r > 0 \\ 
 a_i \varphi^{-1}(a_{i-1})^{-1} \varphi^{-2}(a_{i-1})^{-1} \cdots \varphi^{r}(a_{i-1})^{-1} & \textrm{ when }  r<0. 
 \end{cases} \]
Furthermore, these expressions are reduced words.
  \end{lemma}

One sees that 
the second expression is reduced by appealing again to the fact
$\varphi^{-1}$ is positive with respect to the basis $\{a_1 , a_2^{-1}, a_3 , \ldots , a_m^{\pm 1}\}$.

\begin{prop}\label{prop:lower bound distortion}
Suppose $i \in \set{2, \ldots, m}$ and $\pi$ is a piece of
strict rank-$i$ that is not fixed by $\varphi$.  Then $|\varphi^r(\pi)|_F \sim \abs{r}^{i-1}$.  More precisely,  for   the constants  $C_i, D_i>0$ of Lemma~\ref{lem:growth of a_i}, for all $r \neq 0$, 
	\[\abs{\varphi^r(\pi)}_F 
	 \ \leq \  D_i \abs{r}^{i-1}\abs{\pi}_F\]
and when $\abs{r}^{i-1} \geq \frac{1}{C_i}\abs{\pi}_F$, 
$$\abs{\varphi^{r}(\pi)}_F \ \geq \ \left(C_i^{\frac{1}{i-1}}\abs{r}- {\abs{\pi}_F} ^{\frac{1}{i-1}} \right)^{i-1}.$$
\end{prop}

\begin{proof} 
	The upper bound follows via a simple induction argument on the rank $i$, taking $D_i$  from Lemma~\ref{lem:growth of a_i}.
	We therefore focus on the lower bound.
	
	We have $\pi = a_i^\delta u a_i^{-\delta'}$  a rank-$i$ piece with   $\delta,\delta' \in \{0,1\}$   not both zero and  $u$ a reduced word of  rank at most $i-1$.   We may assume  $\delta =1$, as otherwise we could replace $\pi$ by $\pi^{-1}$.
	
	\case{1} $\delta' =0$ and $\pi = a_i u$.
	
	Let $u=\rho_1 \cdots \rho_p$ be the piece decomposition of $u$.
	Our argument will be that any cancellation between the images of $a_i$ and $u$ under iterated applications of $\varphi$ or $\varphi^{-1}$ will occur within the first $p$ applications.
	After this, there is no further cancellation, so the length of $\pi$ will eventually have growth rate at least that of $a_i$ under iterated applications of $\varphi$ or $\varphi^{-1}$, which will lead  to the required lower bound.
	
	\case{1a} $r>0$.
	
	If the first letter of $u$ is $a_{i-1}$, then there is no cancellation between $\varphi^r(a_i)$ and $\varphi^r(u)$ for any $r>0$, since $\varphi^r(u)$ has first letter $a_{i-1}$ and $\varphi^r(a_i)$ is a positive word.

	So suppose, on the other hand,  that the first letter of $u$ is not $a_{i-1}$,  and that there is cancellation between $\varphi(a_i)$ and $\varphi(u)$.
	Then we may write $\rho_1 = va_i^{-\varepsilon}$, for $\varepsilon\in\{0,1\}$  and $v$ a word of rank at most $i-2$ (if $i=2$ then $v$ is the empty word).
	Then $\varphi(a_iva_{i-1}^{-\varepsilon}) = a_i a_{i-1} \varphi(v)a_{i-2}^{-\varepsilon}a_{i-1}^{-\varepsilon}$ (if $i=2$, we read $a_0$ as the empty word).
	But $\varphi(v)a_{i-2}^{-\varepsilon}$ has rank strictly less than $i-1$, so in order for there to be cancellation between $\varphi(a_i)$ and $\varphi(u)$ we must have $\varepsilon = 1$ and $a_{i-1} \varphi(v)a_{i-2}^{-1}a_{i-1}^{-1}=1$. 
	That is,
	$$
	\varphi(\pi) \ = \  a_i \varphi(\rho_2)\cdots \varphi(\rho_p).$$
	We repeat the argument, and conclude that after $k\leq p$ steps we will reach a situation where we have 
	$$
	\varphi^k(\pi) \ = \  a_i \varphi^k(\rho_{k+1})\cdots \varphi^k(\rho_p)$$
	and there will be no cancellation between images of $a_i$ and $\varphi^k(\rho_{k+1})$ on further applications of $\varphi$.
	(If $k=p$ we understand that $\rho_{k+1}\cdots \rho_p$ is the empty word.)
	Hence, for $r \geq k$, we may write
	$$
	\varphi^r(\pi) \ = \  \varphi^{r-k}(a_i) \varphi^r(\rho_{k+1} \cdots \rho_p)$$
	where there is no cancellation in the right-hand side.
	In particular, this gives
	$$
	\abs{\varphi^r(\pi)}_F \ \geq \  \abs{\varphi^{r-k}(a_i)}_F  \ \geq \  C_i (r-k)^{i-1}$$
	by Lemma~\ref{lem:growth of a_i}.
	
	We complete this case by bounding $k$.
	Indeed, since $\varphi^k(a_i\rho_1 \cdots \rho_k) = a_i$, we get $\abs{\varphi^{-k}(a_i)}_F \leq \abs{\pi}_F$.
	Then an application of Lemma~\ref{lem:growth of a_i} gives $C_ik^{i-1} \leq \abs{\pi}_F$, which implies the required lower bound of $\abs{\varphi^r(\pi)}_F$.
	
	\case{1b} Assume $r<0$.
	
	The argument is broadly similar to Case 1a.
	The key observation   is that there will be no cancellation between images of $a_i$ and $u$ under applications of $\varphi^{-1}$ whenever the first piece of $u$ is of the form $va_{i-1}^{-\delta}$, for $\delta \in \{0,1\}$.
	By induction, the last letter of $\varphi^{r}(a_i)$ as a reduced word on $\{a_1,\ldots,a_m\}$ is $a_{i-1}^{-1}$ for all $r<0$.   (The base case, $r=-1$, is from \eqref{eq:phi inverse}.)
	So the only way we can ever have cancellation between $\varphi^{r}(a_i)$ and $\varphi^{r}(u)$  is if $\varphi^r(\rho_1)$ begins with $a_{i-1}$.
	But since the type of a piece is preserved under applications of $\varphi$, if $\rho_1 = va_{i-1}^{-\delta}$, then this will never occur.
	
	So we may assume $\rho_1 = a_{i-1} v a_{i-1}^{-\delta}$, for $\delta \in \{0,1\}$ and $v$ is a word of rank at most $i-2$.
	Then 
	\[\varphi^{-1}(a_i\rho_1)
	\ = \  \varphi^{-1}(a_i a_{i-1})\varphi^{-1}(va_{i-1}^{-\delta})
	\ = \  a_i \varphi^{-1}(va_{i-1}^{-\delta})
	.\]
	In particular, if $\varphi^{-1}(va_{i-1}^{-\delta})\neq 1$ then further applications of $\varphi^{-1}$ will lead to no cancellation between the images of $a_i$ and of $\varphi^{-1}(va_{i-1}^{-\delta})$, and we can stop.
	Otherwise $\varphi^{-1}(va_{i-1}^{-\delta}) = 1$, which implies that $va_{i-1}^{-\delta} = 1$, and $\rho_1 = a_{i-1}$.
	This gives
	$$\varphi^{-1}(\pi)  \ = \  a_i \varphi^{-1}(\rho_2) \cdots \varphi^{-1}(\rho_p).$$
	Repeating this, we find for some $k\leq p$ that
	$$\varphi^{-k}(\pi)  \ = \  a_i\varphi^{-k}(\rho_{k+1})\cdots \varphi^{-k}(\rho_p)$$
	and there will be no cancellation between images of $a_i$ and $\varphi^{-k}(\rho_{k+1})$ after further applications of $\varphi^{-1}$. 
	(If $k=p$ we understand that $\rho_{k+1}\cdots \rho_p$ is the empty word.)
	As above we then yield, for $r<-k$
	$$\abs{\varphi^r(\pi)}_F  \ \geq  \  \abs{\varphi^{r+k}(a_i)}_F  \ \geq \  C_i(\abs{r}-k)^{i-1}$$
	by Lemma~\ref{lem:growth of a_i}, and $C_ik^{i-1}\leq \abs{\pi}_F$, completing this case.
	
	\case{2}   $\delta' = 1$  and $\pi = a_i u a_i^{-1}$.

	We can apply the arguments from Case 1 to both ends of $u$ and it is not hard to see that the same conclusion is reached.
	The key point  is that if $u$ is completely cancelled out under iterated applications of $\varphi^{\pm 1}$, then the cancellation cannot reach the middle of $u$ simultaneously, meaning that either after several applications of $\varphi$ we obtain $a_i\varphi^k(a_i^{-1})$, or $\varphi^k(a_i)a_i^{-1}$. The length of these grow as required on further applications of $\varphi^{\pm 1}$.
\end{proof}

\subsection{Tools for handling the distortion}\label{subsec:tools for distortion}
 
Recall that for $h \in H$,  $\abs{ h }_H$ denotes the length of the shortest word on $a_1^{\pm 1}, \ldots, a_m^{\pm 1}, s^{\pm 1}$ representing $h$ in $H$.

\begin{proposition}\label{prop:fbc subword length}
	   Suppose $h \in H$ has normal form $\tilde{u}s^p$.
	   If $u'$ is a subword of $\tilde{u}$, then $\abs{u'}_H \leq (2m+1) \abs{h}_H$.
\end{proposition}

\begin{remark}\label{r:corridors}
 This proposition does not rely on any special properties of $\varphi$: 
with a change of constant, it holds for any free-by-cyclic group $M=F\rtimes_\psi\Z$.
We shall sketch a geometric proof of this more general fact that will allow the reader
familiar with van Kampen diagrams to skip the algebraic proof that follows.
This geometric argument assumes that the reader is familiar with 
the use of $s$-corridors (as used in \cite{BridsonGroves}, for example).
 
Let $w$ be a shortest word in the generators $\{a_1,\dots,a_n,s\}$
that equals $h$ in $M$, and consider a least-area 
van Kampen diagram $\Delta$ with boundary label $w^{-1}\tilde{u}s^p$. 
This diagram is a union of its $s$-corridors 
and each point along the
side of a corridor is a distance at most $C$ from a point along the other side,
where $C$ is a constant that depends on $\psi$. 

Let
$x$ and $y$ be the endpoints of the arc in $\partial\Delta$
labelled $u'$. It suffices to argue 
that $x$ can be connected to $y$ by a path in the 1-skeleton of $\Delta$
that has length at most $2pC +|w|$. To construct such a path, 
observe that every vertex $z$ on the
arc of $\partial\Delta$ labelled $\tilde{u}$
can be connected to a vertex on the arc $A$ of $\partial\Delta$
labelled $w$ by crossing at most $p$ of the $s$-corridors, i.e.
the corridors emanating from the arc of $\partial\Delta$ labelled $s^p$. Thus
there is a path $\alpha_z$ of length at most $pC$
from $z$ that ends on $A$. The desired path from $x$ to $y$ is obtained
by following $\alpha_x$ then proceeding along $A$ to the endpoint
of $\alpha_y$ before returning along $\alpha_y$.
\end{remark}
 
To aid the intuition of readers who wish to persist with the algebraic proof,
we present an example.
 
\begin{Example}
	Suppose $h$ is represented by the word $u = sa_6a_5^{-1}s^{-2}a_5s^2a_3$.
	Advance the $s$ at the lefthand end  through   $u$ until it  cancels with  the first $s^{- 1}$,   applying  $\varphi^{- 1}$ to the letters $a_i^{\pm 1}$ it passes, to get 
	$$u_1 \  := \  (a_6a_4a_2a_1^{-1}a_3^{-1}a_5^{-1}) (a_4a_2a_1^{-1}a_3^{-1}a_5^{-1}) s^{-1}a_5s^2a_3$$
satisfying $u = u_1$ in $H$.  Then advance the $s^{-1}$ likewise until it cancels with the $s$ to get 	
		$$u_2 \  := \  (a_6a_4a_2a_1^{-1}a_3^{-1}a_5^{-1}) (a_4a_2a_1^{-1}a_3^{-1}a_5^{-1}) (a_5a_4) sa_3$$
	satisfying $u_1 = u_2$ in $H$.  Then advance the remaining $s$ to the right end to get 		
	$$u_3 \  := \  (a_6a_4a_2a_1^{-1}a_3^{-1}a_5^{-1}) (a_4a_2a_1^{-1}a_3^{-1}a_5^{-1}) (a_5a_4) (a_3 a_1a_2^{-1})$$ satisfying $u_2 = u_3s$ in $H$.  Then $u = u_3s$ in $H$, and $\tilde{u}$ is the reduced version of $u_3$.	Suppose $u' = a_3^{-1}a_4a_3$, a subword close to the right-hand end of $\tilde{u}$.
	Let $u_3' = a_3^{-1}a_5^{-1}a_5a_4a_3$, the subword of $u_3$ that freely reduces to $u'$.
	If we take $u_2' = a_3^{-1}a_5^{-1} a_5a_4$, then $u_3' = u_2' sa_3$; if we take
	$u_1' = a_3^{-1}a_5^{-1} s^{-1}a_5$, then $u_2' = u_1' s$; and if we take
	$u_0' = s^{-1}a_5$, then $u_1' =a_3^{-1}a_5^{-1}u_0'$.
	In particular, $u_0'$ is a subword of $u$, and it is obtained from $u_3'$ by pre- and post-multiplying by a   number (bounded by the exponent sum of $s$ in $u$, and therefore by $\abs{h}_H$ if $u$ is of minimal length) of short words (the images of letters $a_i^{\pm 1}$ under $\varphi^{\pm 1}$, with possibly an $s$ added at the beginning or end).   
	
	So there is a word that equals $u'$  in $H$ and whose length can be bounded from above by the length of a subword of $u$ plus the sum of the lengths of these short words.  The strategy of the following proof is to bound $\abs{u'}_H$ accordingly.
\end{Example}

\begin{proof}[Proof of Proposition~\ref{prop:fbc subword length}]
Suppose $u$ is a word on $a_1^{\pm 1}$, \ldots, $a_m^{\pm 1}, s^{\pm 1}$ representing $h$.
On account of the free-by-cyclic structure of $H$  and the fact that $u  =  \tilde{u}s^{p}$  in $H$, there is a sequence of words $u_0,  \ldots, u_r$ and integers $p_0, \ldots, p_r$  with the following properties (for all $i$):
\begin{itemize}
\item $r \leq \ell(u)$,
\item $u_0  =  u$ (as words) and $p_0=0$, 
\item $u_r$ freely reduces to $\tilde{u}$ and $p_r = p$,
\item $u =  u_i s^{p_i}$ in $H$ (and so the exponent sum of the $s^{\pm 1}$ in $u_i s^{p_i}$ is equal to that of $u$),
\item  there are letters $x_1, \ldots, x_n \in \set{a_1, a_1^{-1}, \ldots, a_m, a_m^{-1}}$  (which depend on $i$)   such that,  as words, either
\begin{enumerate}
\item   $u_{i} = \alpha_i s^{\pm 1} x_1 \cdots x_n s^{\mp 1} \beta_i$ and   $u_{i+1} = \alpha_i  \varphi^{\mp 1}(x_1) \cdots \varphi^{\mp 1}(x_n)   \beta_i$ 
 and $p_{i+1} = p_i$,  or
\item  $u_{i} = \alpha_i s^{\pm 1} x_1 \cdots x_n$ and     $u_{i+1} = \alpha_i  \varphi^{\mp 1}(x_1) \cdots \varphi^{\mp 1}(x_n)$
and $p_{i+1} = p_i \pm 1$, 
\end{enumerate}
for some words $\alpha_i$ and $\beta_i$.  
  (In the first case $u_i = u_{i+1}$ in $H$.  In the second, $u_i s^{\mp 1} = u_{i+1}$ in $H$.)   
  The words $u_i$ need not be reduced.
\end{itemize}

Suppose  $u'_{i+1}$ is a subword of $u_{i+1}$.  We claim that there is a subword $u'_{i}$  of $u_{i}$ and  there are  words $\mu_{i+1}$ and $\lambda_{i+1}$ with $\ell(\mu_{i+1}), \ell(\lambda_{i+1}) \leq m$ such that    $\mu_{i+1} u'_{i} \lambda_{i+1} = u'_{i+1}$ in $H$.    
The details of the proof of this depend on which of   cases (1) and (2)  applies and  how $u'_{i}$ is positioned in relation to the various subwords.  For instance suppose we are in the first case, so that  $u_{i} = \alpha_i s^{\pm 1} x_1 \cdots x_n s^{\mp 1} \beta_i$ and   $u_{i+1} = \alpha_i  \varphi^{\mp 1}(x_1) \cdots \varphi^{\mp 1}(x_n)   \beta_i$ and suppose $u'_{i+1} = \alpha'_i  \varphi^{\mp 1}(x_1) \cdots \varphi^{\mp 1}(x_j) \gamma$ where $\alpha'_i$ is a suffix of $\alpha_i$ and $\gamma$ is a prefix of    $\varphi^{\mp 1}(x_{j+1})$. 
Then taking $u'_i =  \alpha'_i  s^{\pm 1} x_1 \cdots x_j$, the result holds with $\mu_{i+1}$ the empty word and $\lambda_{i+1}= s^{\mp 1} \gamma$.  (The length of $\gamma$ is strictly less than  $\ell(\varphi^{\mp 1}(x_{j+1}))$, which  is at most $m$---see equation~\eqref{eq:phi inverse}.) 
The other cases are similar.

Take $u'_r$ to be a subword of $u_r$ which freely reduces to $u'$.  As per the previous paragraph obtain $u'_{r-1}$, \ldots, $u'_0$ and $\mu_r, \ldots, \mu_1$  and   $\lambda_r, \ldots, \lambda_1$ such that $\mu_r \cdots \mu_1 u'_0 \lambda_1 \cdots \lambda_r = u'_r = u'$ in $H$.    But $\ell( \mu_r \cdots \mu_1 u'_0 \lambda_1 \cdots \lambda_r ) \leq (2m+1) \ell(u)$  since $u'_0$ is a subword of $u$, $r \leq \ell(u)$ and  $\ell(\mu_{i}), \ell(\lambda_{i}) \leq m$ for all $i$.   

So when $u$ is a minimal length word representing $h$ in $H$, we get our result.
\end{proof}

We introduce some notation. If $w$ is the reduced word representing $h \in F$, then $i= \rank(w)$ is the  maximum $i$ such that there is a letter $a_i^{\pm 1}$ in $w$, and then $\pieces{h}$ denotes the number of pieces in the rank-$i$ piece decomposition of $w$.   Further, for a word $w$ and a letter $a$,  the
number of occurrences  of  $a$ in $w$ plus the 
number of  $a^{-1}$ is  $\wt_a(w)$, and  the number of occurrences of $a$ minus the number of $a^{-1}$ is  $\exp_{a}(w)$.  

\begin{lemma}\label{lem:piece bound}
	Suppose $h \in H$ is expressed in  normal form as $\tilde{u}s^r$, where $\tilde{u}$ is a reduced word on $a_1^{\pm 1}, \ldots, a_m^{\pm 1}$ and $r \in \Z$. 	Then   
	$\pieces{\tilde{u}} \leq \abs{h}_H$.   
\end{lemma}

\begin{proof}
Let $v_m$ be a geodesic word on $a_1^{\pm 1}, \ldots, a_m^{\pm 1}, s^{\pm 1}$ representing $h$ in $H$.  So $\ell(v_m) = \abs{h}_H$.

The \emph{shuffling} moves $sa_i \mapsto   \varphi^{-1}(a_i) s$  and  $s^{-1}a_i \mapsto \varphi(a_i) s^{-1}$ transform  $v_m$ to a word $u_m s^r$ 
so that $\tilde{u}$ is the freely reduced form of $u_m$.   Since these moves do not create or remove them,  the $a_m^{\pm 1}$ in $u_m$ correspond with those in $v_m$ in number, sign, and relative location.

Say that a subword $\sigma$ in $v_m$ is \emph{superfluous} if it  has the form  $a_m  \tau  a_m^{-1}$ or $a_m^{-1} \tau  a_m$ for some word $\tau$ on $a_1^{\pm 1}, \ldots, a_m^{\pm 1}$ and the $a_m^{\pm 1}$ and $a_m^{\mp 1}$ that bookend $\sigma$, correspond to an $a_m^{\pm 1}$ and an $a_m^{\mp 1}$ that bookend  a subword $\overline{\sigma}$ in $u_m$ that freely reduces to the identity. Given such a $\sigma$, write $v_m = \sigma_0 \sigma \sigma_1$ (as words), let $\lambda= \exp_s(\tau)$, and let $\mu = \exp_{s}(\sigma_0)$. There exists  a word $\xi$ on $a_1^{\pm 1}, \ldots, a_m^{\pm 1}$ arising  in the following three cases as follows: 
\begin{enumerate}
\item \label{case 1 of shuffling} If  $\sigma = a_m  \tau  a_m^{-1}$, then the shuffling moves give that  $\sigma = a_m \xi  a_m^{-1}s^{\lambda}$  in $H$. Further, $\varphi^{-\mu}(a_m  \xi  a_m^{-1})  =   \overline{\sigma}$.  Define $\sigma' = s^{\lambda}$.
\item \label{case 2 of shuffling} If  $\sigma = a_m^{-1}  \tau  a_m$ and $\lambda \leq 0$, then they give that  $\sigma = a_m^{-1} \xi  a_m a_{m-1}^{-\lambda} s^{\lambda}$  in $H$. Further, $\varphi^{-\mu}( \xi  )  =   \overline{\sigma}$.  Define $\sigma' = a_{m-1}^{-\lambda}s^{\lambda}$.
\item \label{case 3 of shuffling} If  $\sigma = a_m^{-1}  \tau  a_m$ and $\lambda > 0$, then moves $a_i s \mapsto  s \varphi(a_i)$  and  $a_i s^{-1}  \mapsto s^{-1} \varphi^{-1}(a_i)$ give that  $\sigma = s^{\lambda} a_{m-1}^{-\lambda} a_m^{-1} \xi  a_m$ in $H$. Further, $\varphi^{-\mu-\lambda}(  \xi  )  =   \overline{\sigma}$. Define $\sigma' = s^{\lambda} a_{m-1}^{-\lambda}$.
\end{enumerate} 
In each case  $\xi = 1$, because $\overline{\sigma}   =   1$, and so  $\sigma = \sigma'$  in $H$.

Let $i = \rank(\tilde{u})$.    If $i <m$, then all the $a_m^{\pm 1}$ in $u_m$  cancel away on free reduction of $u_m$.  So there exists a family of pairwise disjoint superfluous subwords of $v_m$ which together contain every  $a_m^{\pm 1}$ in $v_m$.    Let $v_{m-1}$ be the word obtained from   $v_m$ by replacing each of these superfluous  subwords $\sigma$  by the corresponding $\sigma'$ described above.    Then  $v_{m-1}  =  v_m$  in $H$ and 
$$\begin{array}{rl}
 \rank(v_{m-1}) & \!\! \leq \ m-1, \\  
 \wt_s(v_{m-1}) & \!\! \leq \ \wt_s(v_{m}), \\
 \wt_{a_j}(v_{m-1})  & \!\! \leq \  \wt_{a_j}(v_{m}) \text{ for } j=1, \ldots, m-2, \\ 
 \wt_{a_{m-1}}(v_{m-1}) & \!\! \leq \ \wt_{a_{m-1}}(v_{m}) +  \wt_s(v_{m}), 
 \end{array}$$   
where the final inequality holds because the $a_{m-1}^{-\lambda}$ inserted in all instances of cases \eqref{case 2 of shuffling} and \eqref{case 3 of shuffling} contribute a total of no more than $\wt_s(v_{m})$ letters  $a_{m-1}^{\pm 1}$.  

If $i <m-1$, then, because there are   no $a_m^{\pm 1}$ letters in $v_{m-1}$, we can obtain a word $v_{m-2}$ from $v_{m-1}$ in the same manner that we obtained  $v_{m-1}$ from $v_{m}$ and subject to the same  inequalities as displayed above, but with $m$ decremented by $1$.  Repeat until  arriving at $v_i$.  Then $\rank(v_{i}) =i$ and 
\begin{equation} \label{sum of weights bound} 
\wt_{a_1}(v_{i}) + \cdots + \wt_{a_i}(v_{i}) \  \leq  \ \wt_{a_1}(v_{m}) + \cdots + \wt_{a_i}(v_{m}) + \wt_{s}(v_{m}) \ = \ \ell(v_m).
\end{equation}

Now, $v_i$ freely equals $w_0 s^{\alpha_1} w_1 s^{\alpha_2} w_2 \cdots s^{\alpha_k} w_k$  for some  reduced words  $w_0, \ldots, w_k$  on $a_1^{\pm 1}, \ldots, a_i^{\pm 1}$  and some non-zero $ \alpha_1, \ldots, \alpha_k \in \Z$.  
For $0 \leq j  \leq k$,  let $\beta_j =  - \alpha_1 - 	\cdots - \alpha_j$, so that $\tilde{u}$  freely equals  $\varphi^{\beta_0}(w_0) \cdots \varphi^{\beta_k}(w_k)$.  The number of pieces in the rank-$i$ decomposition of $w_j$ is at most $\ell(w_j)$ since each piece has at least one letter.  By Lemma~\ref{lem:pieces exercise},  the rank-$i$ decompositions of $w_j$ and of  $\varphi^{\beta_j}(w_j)$ have the same number of pieces. Free reduction between an  $\varphi^{\beta_j}(u_j)$ and the neighbouring $\varphi^{\beta_{j+1}}(u_{j+1})$ can only cause pieces to merge or cancel, so  $\pieces{\tilde{u}} \leq \ell(w_0) + \cdots +  \ell(w_k)$, which is at most $\ell(v_m)$ by \eqref{sum of weights bound}.   
\end{proof}

\begin{cor}\label{cor:rank i distortion}
	For $i=1, \ldots, m$, there exists $K_i >0$ such that for all $g \in F$ of rank $i$,	we have $\abs{g}_F \leq K_i\pieces{g}\abs{g}_H^{i-1}$.
\end{cor}

\begin{proof}
	We  induct on   $i$.
	If $i=1$, then $g = a_1^k$ for some $k\in \Z$, and $\abs{k}=\abs{g}_F = \pieces{g}$.
	
	Now assume $i>1$. Express $g$, viewed as a reduced word on $a_1^{\pm 1}, \ldots, a_i^{\pm 1}$, as a rank-$i$ product of pieces $\pi_1 \cdots \pi_p$.
	Each piece is $\pi_k = a_i^{\delta_k} v_k a_i^{-\varepsilon_k}$ for some $\delta_k,\varepsilon_k\in \{0,1\}$ and $v_k$ a reduced word of rank at most $i-1$.
	By   induction $\abs{v_k}_F \leq K_{i-1} \pieces{v_k} \abs{v_k}_H^{i-2}$.  By 
	Lemma~\ref{lem:piece bound} $\pieces{v_k} \leq  \abs{v_k}_H$. So $\abs{v_k}_F \leq K_{i-1}\abs{v_k}_H^{i-1}$.
	We then get
	\[ \abs{g}_F 
	\ = \  \sum_{k=1}^{p} \abs{\pi_k}_F 
	\ \leq \  \sum_{k=1}^p (K_{i-1}\abs{v_k}_H^{i-1} + 2)
	\ \leq \  \pieces{g} ( K_{i-1} (2m+1)^{i-1}\abs{g}_H^{i-1} + 2)
	\]
	where the last inequality follows from Proposition~\ref{prop:fbc subword length}.
\end{proof}

If $g=s^{-k}a_i^ks^{k} = \varphi^k(a_i^k)$, then $\abs{g}_H\leq 3k$ and
	$g$ is a product of $k$ pieces, each of which is $\varphi^k(a_i)$,
	a positive word whose length grows like a polynomial in $k$ of degree $i-1$,
	as we saw in Lemma~\ref{lem:growth of a_i}.
	Thus the bound in Corollary~\ref{cor:rank i distortion} is sharp 
	in this case (up to constants).

\section{Solving the 0-twisted conjugacy problem} \label{0-twisted section}

 As we saw in Section~\ref{reductions}, the conjugacy relation $uw = wv$ in $H$ amounts to the $\varphi$-twisted conjugacy relation $\tilde{u}\varphi^{-p}(\tilde{w}) = \tilde{w}\varphi^{-r}(\tilde{v})$ in $F$, where we have normal forms $u=\tilde{u}s^p,v=\tilde{v}s^p$, and $w=\tilde{w}s^r$.
 We assume in this section that $p=0$.

Recall that the 0-twisted conjugacy problem asks whether,
given words $\tilde{u}$, $\tilde{v}$ on $a_1^{\pm 1}, \ldots, a_m^{\pm 1}$,  there exist $r \in \Z$ and $\tilde{w} \in F$ such that  
\begin{equation} \label{eq:0 twisted}
\tilde{u}  \tilde{w}    \   =   \   \tilde{w} \varphi^{-r}(\tilde{v}) \ \ \text{in $F$.}
\end{equation}

\begin{prop}[0-twisted conjugacy problem] \label{0 twisted conj problem}
	There exists $A >0$ with the following property.  Suppose $\tilde{u}$ and $\tilde{v}$ are words on $a_1^{\pm 1}, \ldots, a_m^{\pm 1}$.
	\begin{enumerate}\renewcommand{\theenumi}{\Roman{enumi}}
		\item \label{case:0-twisted CL} If there exist $r\in \Z$ and $\tilde{w} \in F$ satisfying \eqref{eq:0 twisted}, then there are such $r$ and $\tilde{w}$ with  
		$\abs{r} + \abs{\tilde{w}}_H \leq A (\abs{\tilde{u}}_H + \abs{\tilde{v}}_H)$.
		\item \label{case:0-twisted coda} If there exists $\tilde{w} \in F$ satisfying \eqref{eq:0 twisted} with $r=0$, then there exists such  $\tilde{w}$ with $\abs{\tilde{w}}_H \leq A (\abs{\tilde{u}}_H + \abs{\tilde{v}}_H)$.
		\item \label{case:0-twisted compl} There is an algorithm that,
		given $\tilde{u}$ and $\tilde{v}$ will determine
		 whether or not there exist $r \in \Z$ and $\tilde{w}\in F$ solving \eqref{eq:0 twisted}, and will exhibit them if they exist.  The running time of this algorithm 
		 is polynomial in $\abs{\tilde{u}}_H+\abs{\tilde{v}}_H$. 	
	\end{enumerate}
\end{prop}

The proof of Proposition~\ref{0 twisted conj problem} uses the following lemma, which determines the form of a `short' solution to \eqref{eq:0 twisted} whenever any solution exists.

\begin{lemma}\label{lem:0 twisted form}
	There exists $B >0$ with the following property.    Suppose  $\tilde{u},  \tilde{v} \in F$  are as in Proposition~\ref{0 twisted conj problem}.  Suppose there exist  $r \in \Z$ and $\tilde{w} \in F$ satisfying \eqref{eq:0 twisted}.
	Then there exist $\tilde{w}_0 \in F$ satisfying 
	\begin{equation} \label{eq:0 twisted short}
	\tilde{u}  \tilde{w}_0  \   = \   \tilde{w}_0 \varphi^{-r}(\tilde{v}) \ \ \text{in $F$,}
	\end{equation} 
	 with $\tilde{w}_0 = UV$, where
	 $U$ is a prefix of $\tilde{u}^{-1}$  and
	 $\varphi^r(V)$ is a suffix of $\tilde{v}$.
	 
	 Furthermore, either  $\tilde{u}  \tilde{w}     =    \tilde{w}  \tilde{v}$ in $F$, or  
	 $\abs{r} \leq B (\abs{\tilde{u}}_H + \abs{\tilde{v}}_H)$.
\end{lemma}
	
\begin{proof}   	
First replace $\tilde{u}$ by a cyclic conjugate $u'$  such that $u'u'$ is reduced (i.e.\ $u'$ is cyclically reduced) and the  rank-$m$ piece decomposition of $u' \, u'$ is the concatenation of two copies of the piece decomposition of $u'$---that is, the rightmost piece in $u'$ does not combine with the leftmost piece in $u'$ to make a single piece in $u' \, u'$.  	
	This is achieved by  conjugating $\tilde{u}$ by a suitable $y$ that is a prefix of   $\tilde{u}^{-1}$, so that $y^{-1} \tilde{u} y = u'$ in $F$.	
	
	Likewise, replace $\tilde{v}$ by a similarly structured cyclic conjugate $v'$. Let $z$ be the prefix of $\tilde{v}^{-1}$ such that $z^{-1} \tilde{v} z = v'$ in $F$.  
	
Lemma~\ref{lem:pieces exercise} gives us that  $\varphi^k(u')$ and  $\varphi^k(v')$ are cyclically reduced for all $k \in \Z$.    
		
	By assumption, we have $r \in \Z$ and $\tilde{w} \in F$ satisfying \eqref{eq:0 twisted}.
	This implies that there is $\tilde{w}_1\in F$ satisfying
	\begin{equation}\label{eq:after cyclic reduction}
	u' \tilde{w}_1 \ = \  \tilde{w}_1 \varphi^{-r} (v') \ \  \text{ in } F.
	\end{equation}
	Since \eqref{eq:after cyclic reduction} is a conjugacy relation in $F$ and  $\varphi^{-r} (v')$ is cyclically reduced,   Lemma~\ref{CP in F lemma} tells us that there is a prefix  $\tilde{w}_2$  of $(u')^{-1}$ such that $u' \tilde{w}_2 = \tilde{w}_2 \varphi^{-r} (v')$ in $F$.
  	Then $\tilde{w}_0 = y \tilde{w}_2 \varphi^{-r}(z^{-1})$ satisfies \eqref{eq:0 twisted short} with $r_0=r$.
	 We then take    $U = y \tilde{w}_2$, which  is a prefix of $\tilde{u}^{-1}$ (since $y$ is, and $\tilde{w}_2$ is a prefix of $(u')^{-1} = y^{-1}\tilde{u}^{-1} y$)   
	and $V = \varphi^{-r}(z^{-1})$.
	 From the definition of $z$, $\varphi^r(V) = z^{-1}$ is a suffix of $\tilde{v}$.  
	
	To complete the proof, we need  to bound $\abs{r}$.  
	
	If $v'$ is fixed by $\varphi$ or $r=0$, then $v'$ is conjugate to $u'$ in $F$ and we have $\tilde{u}  \tilde{w}     =    \tilde{w}  \tilde{v}$ in $F$.
	
	Suppose, then, that $v'$ is not fixed by $\varphi$ and $r \neq 0$.
	Let $i$ be the rank of $v'$.  
	As $v'$ is not fixed by $\varphi$, Lemma~\ref{lem:fixed pieces} tells us that   it has a piece $\pi$ that is itself of rank $i$ and is not fixed by $\varphi$.
	Proposition~\ref{prop:lower bound distortion}  and Corollary~\ref{cor:rank i distortion}  give constants $C_i, K_i >0$ such that either
	\begin{align}
	\label{eq:0-twisted bound r 1}
		C_i \abs{r}^{i-1} &\ < \ \abs{\pi}_F \ \leq \ K_i \abs{\pi}_H^{i-1},
	\end{align}
or
	\begin{align}
	\label{eq:0-twisted bound r 2}
	C_i\left(\abs{r}- \left( \frac{\abs{\pi}_F}{C_i} \right)^{\frac{1}{i-1}} \right)^{i-1}
	&\ \leq \ \abs{\varphi^{-r}(\pi)}_F 
	\ \leq \ K_i \abs{\varphi^{-r}(\pi)}_H^{i-1}.
	\end{align}
Since $\pi$ is a subword of $v'$, which is a subword of $\tilde{v}$, Proposition~\ref{prop:fbc subword length} gives us $\abs{\pi}_H \leq (2m+1)\abs{\tilde{v}}_H$.
	Then by Corollary~\ref{cor:rank i distortion} we get $\abs{\pi}_F \leq K_i(2m+1)^{i-1}\abs{\tilde{v}}_H^{i-1}$.
	Meanwhile, since $u'$ and $\varphi^{-r}(v')$ are cyclically reduced and conjugate in $F$, it follows that $\varphi^{-r}(\pi)$ is a subword of $u'u'$. 
	Since the piece decomposition of $u'u'$ consists of the concatenation of two copies of that of $u'$, we must have that $\varphi^{-r}(\pi)$ is a subword of $u'$.	
	So $\abs{\varphi^{-r}(\pi)}_H \leq (2m+1)\abs{\tilde{u}}_H$ by Proposition~\ref{prop:fbc subword length}.
Both \eqref{eq:0-twisted bound r 1} and \eqref{eq:0-twisted bound r 2} lead to
	\[
	\abs{r}  \ \leq \  \left( \frac{K_i}{C_i} \right)^{\frac{1}{i-1}}  (2m+1) (\abs{\tilde{u}}_H + \abs{\tilde{v}}_H),
	\]
	showing that a suitable $B>0$ exists.  
\end{proof}

\begin{proof}[Proof of Proposition~\ref{0 twisted conj problem}]
We begin by establishing the length bounds in \eqref{case:0-twisted CL} and \eqref{case:0-twisted coda}.

Since in Lemma~\ref{lem:0 twisted form} the value of $r$ does not change between \eqref{eq:0 twisted} and \eqref{eq:0 twisted short},   case~\eqref{case:0-twisted coda} of Proposition~\ref{0 twisted conj problem} holds.
Indeed, we need only use Proposition~\ref{prop:fbc subword length} to bound $\abs{U}$ and $\abs{V}$.
Similarly, the bound on $\abs{r}$ and the form of $\tilde{w}_0$ give \eqref{case:0-twisted CL}.

Next we consider the complexity of the algorithms solving the 0-twisted conjugacy problem and its search variant.

Lemma~\ref{lem:0 twisted form} tells us that if a solution exists, then there is a solution of a particularly nice form.
On input $\tilde{u}$ and $\tilde{v}$, we   list all pairs $(\tilde{w},r)$, where $r $ is an integer satisfying $\abs{r} \le B(\abs{\tilde{u}}_H + \abs{\tilde{v}}_H)$, and $\tilde{w}$ has the form $UV$, with $U$ a prefix of $\tilde{u}^{-1}$ and $V=\varphi^{-r}(\hat{V})$, where $\hat{V}$ is a suffix of $\tilde{v}$.
It is not hard to see that the number of such pairs \blue{is} polynomially bounded in terms of $\abs{u}_H + \abs{v}_H$. So a  search through this list for a solution to \eqref{eq:0 twisted}   can be completed in polynomial time. 
If none is found, we conclude that no solution exists.
\end{proof}

\section{Preserved prefixes} \label{prefixes section}

Recall that $uw=wv$ in $H$ amounts to  $\tilde{u}\varphi^{-p}(\tilde{w}) = \tilde{w}\varphi^{-r}(\tilde{v})$ in $F$, where $u=\tilde{u}s^p,v=\tilde{v}s^p$, and $w=\tilde{w}s^r$ are the normal forms.
We now assume $p\ne 0$.
In the specific case when $\tilde{u}$ is the empty word, 
$\tilde{w}$ is a concatenation of a    prefix of  $\varphi^{-p}(\tilde{w})$ with a subword of $\varphi^{-r}(\tilde{v})^{-1}$. 
This points to the fact that, in order to understand the length (or structure) of $\tilde{w}$ (and hence $w$), it is important to understand the length (or structure) of the longest common prefix of $\tilde{w}$ and $\varphi^{-p}(\tilde{w})$.

We begin with the instance where $\tilde{w}$ is a single piece   of a type that behaves well with regards to common prefixes.  This will feed  into the general case in Corollary~\ref{cor:fbc shared prefix of words} below.

If $\pi$ is $a_2 a_1^q$ for some $q \geq 0$, then (assuming $r \geq 0$) the length $1+q$ of the  longest common prefix  $a_2 a_1^q$  of $\pi$ and $\varphi^r(\pi)=a_2 a_1^{q+r}$ can be arbitrarily  large compared to $\abs{\pi^{-1}\varphi^r(\pi)}_H = \abs{a_1^r}_H=r$. 
The same can be said when $\pi$ is $a_1^q$ or $a_1^q a_i^{-1}$ for any $i \geq 2$ and $q \in \Z$.  In contrast, for other types of pieces, the form of the longest common prefix is constrained in a manner that strongly restricts its length:

\begin{lemma}\label{lem:fbc shared prefix}
  For $3 \leq i \leq m$, there exists $B_i >0$ with the following property.   Suppose $r > 0$ and that   $\pi$ is a rank-$i$ piece whose first letter is $a_i$.  Then the longest common prefix $L$ of $\pi$ and $\varphi^r(\pi)$ is a concatenation $L = \Lambda_1 \Lambda_2$ of words, where 
\begin{itemize}
\item $\Lambda_1$  is a prefix of  $\varphi^k(a_i)$ for some $k \in \Z$ satisfying $\abs{k} \leq B_i \left(\abs{\pi^{-1}\varphi^r(\pi)}_H + \abs{r}\right)$, 
\item $\Lambda_2$ is a subword of $\varphi^r(a_i^{-1})$.
\end{itemize} 
\end{lemma}

\begin{proof} 
	We may   assume that $\ell(L) \geq 2$, else $L = a_i$ and the result is immediate with $L= \Lambda_1$ and $k=0$.
	
	\case{1} $\pi = a_i u$ for a word $u$ of rank less than $i$.
		  
	We claim that either there is no cancellation  (as in Figure~\ref{fig:fbc commonprefix1b})  between $\varphi^r(a_i)$ and $\varphi^r(u)$,
	or there is    \emph{complete cancellation} (as in Figure~\ref{fig:fbc commonprefix1a}) by which we mean that $\varphi^r(a_i) \varphi^r(u)$, freely reduces to $a_i$ times a suffix of $\varphi^r(u)$. After all, if there is not  complete cancellation, then the first two letters $a_i a_{i-1}$ of $\varphi^r(a_i)$ are not cancelled away on free reduction of  $\varphi^r(a_i)\varphi^r(u)$.  As   $\ell(L) \geq 2$, these are also the first two letters of $\pi = a_i u$, and so  the first letter of $u$ is $a_{i-1}$.  But then the first letter of $\varphi^r(u)$
must also be $a_{i-1}$,   and as $\varphi^r(a_i)$ is a positive word, there is no cancellation between it and $\varphi^r(u)$.

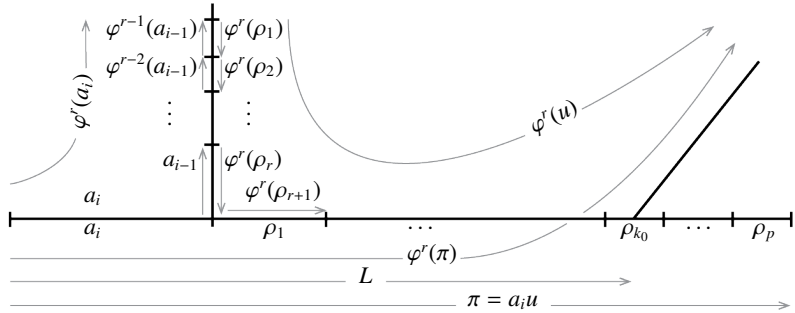
\begin{figure}[ht]
\begin{overpic}[
scale=1.0,unit=1mm]{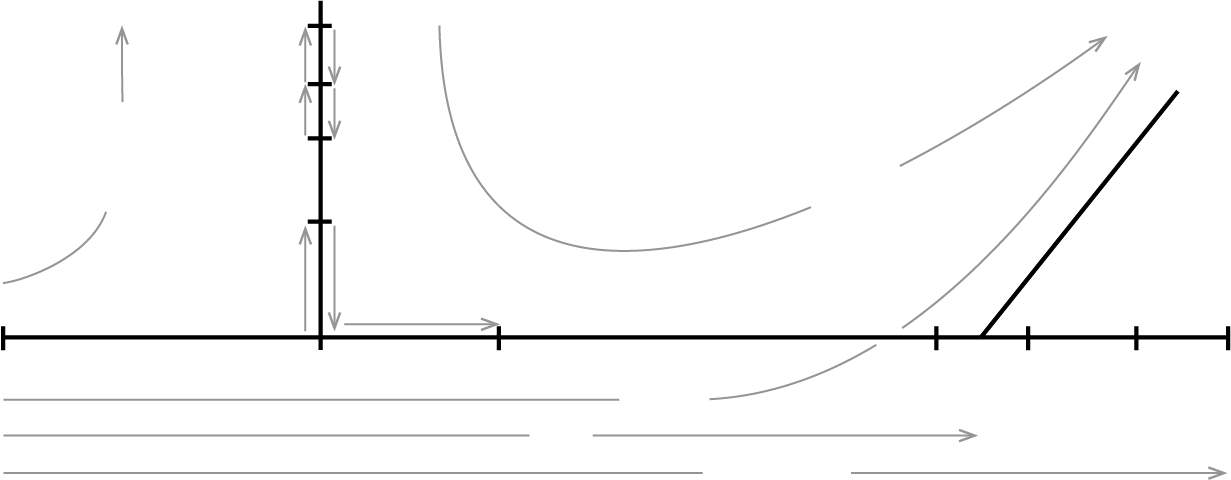}
 \put(61,.5){\small{$\pi = a_i u$}}     
 \put(8,24){\small{\rotatebox{85}{$\varphi^r(a_i)$}}}     
 \put(46.5,3.5){\small{$L$}}     
  \put(69,23.5){\small{\rotatebox{25}{$\varphi^r(u)$}}}     
 \put(53,6.5){\small{$\varphi^r(\pi)$}}     

 \put(10,14.3){\small{$a_i$}}     

 \put(10,10){\small{$a_i$}}     
  \put(34,10){\small{$\rho_1$}}    
 \put(53,10){\small{$\cdots$}}    
 \put(81.5,10){\small{$\rho_{k_0}$}}    
 \put(90,10){\small{$\cdots$}}    
 \put(99,10){\small{$\rho_p$}}    

 \put(28.7,36.5){\small{$\varphi^r(\rho_1)$}}    
 \put(28.7,31.5){\small{$\varphi^r(\rho_2)$}}    
 \put(31.5,24.5){\small{$\vdots$}}    
 \put(28.7,19){\small{$\varphi^r(\rho_r)$}}    

 \put(31.5,15){\small{$\varphi^r(\rho_{r+1})$}}

 \put(13,36.5){\small{$\varphi^{r-1}(a_{i-1})$}}    
 \put(13,31.5){\small{$\varphi^{r-2}(a_{i-1})$}}    
 \put(21,24.5){\small{$\vdots$}}    
 \put(20.5,19){\small{$a_{i-1}$}}

\end{overpic}
 \caption{Cancellation as per Case 1a of the proof of Lemma~\ref{lem:fbc shared prefix}.}
  \label{fig:fbc commonprefix1a}
\end{figure}

Let $\rho_1 \cdots \rho_p$ be the rank-$(i-1)$ decomposition of $u$ into pieces. By Lemma~\ref{lem:pieces exercise},  $\varphi^r(\rho_1) \cdots \varphi^r(\rho_p)$	 is the rank-$(i-1)$ piece decomposition of $\varphi^r(u)$. 
	Choose $k_0$ so that
	\begin{equation} 
	L \ = \  a_i \rho_1 \cdots \rho_{k_0-1} \rho_{k_0}', \label{L set up}
	\end{equation}
	where $k_0$ is chosen so that $\rho_{k_0}'$ is a non-empty prefix of $\rho_{k_0}$.
		
	\case{1a} Complete cancellation.
	
We will show that  $a_i \rho_1 \cdots  \rho_{k_0} = \varphi^{-k_0}(a_i)$, and so $L$ is a prefix of this and the result will hold with $L = \Lambda_1$ and $\Lambda_2$ the empty word.
	
	Lemma~\ref{lem:fbc image of a_i} tells us that $\varphi^r(a_i)   =   
a_i \,  a_{i-1}  \varphi(a_{i-1}) \cdots \varphi^{r-1}(a_{i-1})$.  From this we can read off the first $r$ pieces of $\varphi^r(u)$ on account of the `complete cancellation' between $\varphi^r(a_i)$ and  $\varphi^r(u)$:  for $k=1, \ldots, r$ we have  	$\varphi^r(\rho_k) = \varphi^{r-k}(a_{i-1})^{-1}$, or equivalently
	 \begin{equation}
	 \rho_k \ = \ \varphi^{-k}(a_{i-1})^{-1}. \label{what rho is}
	 \end{equation}
	
After `complete cancellation' 
$\varphi^r(\pi) = a_i \varphi^r(\rho_{r+1}) \cdots \varphi^r(\rho_{p})$.  
As $L$ is also a prefix of $\varphi^r(\pi)$,   
	$$L  \ = \ 
	a_i \varphi^r(\rho_{r+1}) \cdots \varphi^r(\rho_{r+k_0 -1}) \rho_{k_0}'$$ 
where, comparing with  \eqref{L set up},    for $k=1, \ldots, k_0-1$ we have $\rho_k = \varphi^{r}(\rho_{k+r})$, or equivalently $\rho_{k+r} = \varphi^{-r}(\rho_k)$. 
By induction, we can extend \eqref{what rho is} to $k = 1, \ldots, k_0 {+r} -1$.
	This tells us in particular that   
		 \begin{equation}
	\rho_{k_0}   \ =  \   \varphi^{-k_0}(a_{i-1})^{-1}, \label{k0 case}
			 \end{equation}
since $r>0$, and 
	$$a_i \rho_1 \cdots \rho_{k_0}  \ = \  a_i \varphi^{-1}(a_{i-1})^{-1} \cdots \varphi^{-k_0}(a_{i-1})^{-1},$$	which equals, as a word, $\varphi^{-k_0}(a_{i})$ by Lemma~\ref{lem:fbc image of a_i}.  
It follows that $L$ is a prefix of $\varphi^{-k_0}(a_i)$.

Next we will give an upper bound on  $k_0$ that will imply an upper bound on $\abs{-k_0+1}$, proving  the condition on $\Lambda_1$.  From \eqref{k0 case}  we get  \begin{equation}
	C_{i}k_0^{i-1}
	\  \leq	\ \abs{\rho_{k_0}}_F \label{first ineq}
	\end{equation}	
by Lemma~\ref{lem:growth of a_i}.  
	As $\rho_{k_0}$ is a single piece, Corollary~\ref{cor:rank i distortion} gives 
		\begin{equation}
 \abs{\rho_{k_0}}_F 
	\ \leq \  K_i\abs{\rho_{k_0}}_H^{i-1}.  \label{second ineq}
		\end{equation} 
	View $\rho_{k_0}$ as a product of $\rho_{k_0}'$ with a subword of $\pi^{-1}\varphi^r(\pi)$.  Then apply Proposition~\ref{prop:fbc subword length} to give 
			\begin{equation}
\abs{\rho_{k_0}}_H  
	\ \leq \ \abs{\rho_{k_0}'}_H + (2m+1)\abs{\pi^{-1}\varphi^r(\pi)}_H.   \label{third ineq}
		\end{equation}
	 To bound $\abs{\rho_{k_0}'}_H$, observe that $\rho_{k_0}'$ is a prefix of $\varphi^r(\rho_{k_0+r})$, and $\rho_{k_0+r}$ is a subword of $\pi^{-1}\varphi^r(\pi)$, since $r>0$.
		So, applying Proposition~\ref{prop:fbc subword length}, we first get
	$$\abs{\rho_{k_0}'}_H 
	\ \leq \ (2m+1)\abs{\varphi^r(\rho_{k_0+r})}_H.
	$$
	Then using that $\varphi^r(\rho_{k_0+r}) = s^{-r} \rho_{k_0+r} s^{r}$, we deduce that 
	$$
	\abs{\rho_{k_0}'}_H \ \leq \ (2m+1)\left(2r+\abs{\rho_{k_0+r}}_H\right).
	$$
	A last application of Proposition~\ref{prop:fbc subword length} then gives
	\begin{equation}\label{eq:bound rho'}
	\abs{\rho_{k_0}'}_H \ \leq \ (2m+1)\left(2r+(2m+1)\abs{\pi^{-1}\varphi^r(\pi)}_H\right).
	\end{equation}
	Together \eqref{first ineq}--\eqref{eq:bound rho'} show $k_0$ is  at most a constant times $\abs{r} + \abs{\pi^{-1}\varphi^r(\pi)}_H$, as required.

\begin{figure}[ht]
\begin{overpic}[
scale=1.0,unit=1mm]{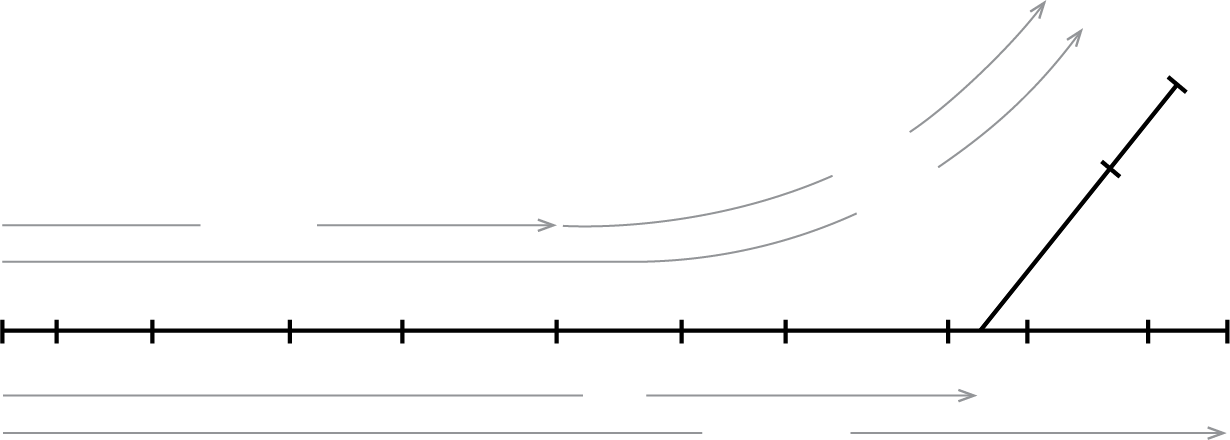}
 \put(61,0.5){\small{$\pi = a_i u$}}     
 \put(18.5,18){\small{$\varphi^r(a_i)$}}     
 \put(51,3.3){\small{$L$}}     
  \put(70.2,23){\small{\rotatebox{30}{$\varphi^r(u)$}}}     
  \put(72.5,19.5){\small{\rotatebox{30}{$\varphi^r(\pi)$}}}     
 \put(1.5,7.3){\small{$a_i$}}    
 \put(7,11.5){\small{$a_{i-1}$}}    
 \put(15,11.5){\small{$\varphi(a_{i-1})$ }}    
 \put(28,11.5){\small{$\cdots$}}    
 \put(35,11.5){\small{$\varphi^{r-1}(a_{i-1})$}}    
 \put(49,11.5){\small{$\varphi^{r}(\rho_1)$}}    
 \put(61,11.5){\small{$\cdots$}}    
 \put(67,11.5){\small{$\varphi^{r}(\rho_{k_0-1-r})$}}    
 \put(7,7.3){\small{$\rho_1$}}    
 \put(17,7.3){\small{$\rho_2$}}    
 \put(28,7.3){\small{$\cdots$}}    
 \put(40,7.3){\small{$\rho_r$}}    
 \put(51,7.3){\small{$\rho_{r+1}$}}    
 \put(61,7.3){\small{$\cdots$}}    
 \put(72,7.3){\small{$\rho_{k_0-1}$}}    
 \put(82,7.3){\small{$\rho_{k_0}$}}    
 \put(90.5,7.3){\small{$\cdots$}}    
 \put(99,7.3){\small{$\rho_{p}$}}    
 \put(91.5,25.5){\small{\rotatebox{48}{$\varphi^r(\rho_{p})$}}}    
\end{overpic}
 \caption{Cancellation as per Case 1b of the proof of Lemma~\ref{lem:fbc shared prefix}.}
  \label{fig:fbc commonprefix1b}
\end{figure}

	\case{1b} No cancellation.    
	
	We will show that   $L$ is a prefix of $\varphi^{k_0}(a_i)$ and that the result will again hold with $L = \Lambda_1$ and $\Lambda_2$ the empty word.
	
	Comparing pieces along  the common prefix $L  = a_i \rho_1 \cdots \rho_{k_0-1}\rho_{k_0}'$ of $\pi = a_i  \rho_1 \cdots \rho_p$ and    $$\varphi^r (\pi) \  = \  \varphi^r(a_i) \varphi^r(\rho_1) \cdots \varphi^r(\rho_p) \ = \    
a_i \,  a_{i-1}  \varphi(a_{i-1}) \cdots \varphi^{r-1}(a_{i-1}) \varphi^r(\rho_1) \cdots \varphi^r(\rho_p),$$ which  is a freely reduced word in this case,  
we claim that 
\begin{equation}\label{eq:case 1b}
\rho_k  \ = \ \varphi^{k-1}(a_{i-1}) \ \  \textrm{for \ $k=1, \ldots, k_0-1$.} 
\end{equation}
For $k \leq \min\{r,k_0-1\}$  we get $\rho_k  = \varphi^{k-1}(a_{i-1})$ immediately, so there is nothing left to show when $r\geq k_0-1$.	When $r<k_0-1$, if $k = r +1, \ldots, k_0-1$ then $\rho_k  = \varphi^r( \rho_{k-r})$, and this gives the claim inductively.
	
We claim that $\rho_{k_0}'$ is a prefix of $\varphi^{k_0-1}(a_{i-1})$. Indeed, if $k_0\leq r$, then this is immediate.
Meanwhile, if $k_0>r$, $\rho_{k_0}'$ is a subword of $\varphi^r(\rho_{k_0-r})$ which equals $\varphi^{k_0-1}(a_{i-1})$ by \eqref{eq:case 1b} since $r>0$. 
	 
It follows that $L$ is a prefix of $a_i a_{i-1} \varphi(a_{i-1}) \cdots \varphi^{k_0-1}(a_{i-1})$, which equals, as a word, $\varphi^{k_0}(a_i)$.

We complete this case by bounding $k_0$.  The process is similar to Case~1a.

	Since $\varphi^r(\rho_{k_0})$ is a subword of $\pi^{-1}\varphi^r(\pi)$, we can obtain a bound on the length of $\rho_{k_0}'$ as follows.
	Firstly, $\abs{\rho_{k_0}'}_H \leq (2m+1)\abs{\rho_{k_0}}_H$ by Proposition~\ref{prop:fbc subword length}.
	Then, using $\rho_{k_0} = s^{r}\varphi^{r}(\rho_{k_0})s^{-r}$, and that $\varphi^{r}(\rho_{k_0})$ is a subword of $\pi^{-1}\varphi^r(\pi)$, we get $\abs{\rho_{k_o}}_H \leq 2\abs{r} + (2m+1)\abs{\pi^{-1}\varphi^r(\pi)}_H$. Thus, we have
	$$\abs{\rho_{k_0}'}_H \ \leq \ (2m+1)\left(2\abs{r} + (2m+1)\abs{\pi^{-1}\varphi^r(\pi)}_H\right).$$

	Since $\varphi^r(\rho_{k_0-r}) = \varphi^{k_0-1}(a_{i-1})$, Lemma~\ref{lem:growth of a_i} and Corollary~\ref{cor:rank i distortion} imply that
	$$C_{i}(k_0-1)^{i-1} \ \leq \  \abs{\varphi^r(\rho_{k_0-r})}_F \ \leq \  K_i \abs{\varphi^r(\rho_{k_0-r})}_H^{i-1}.$$
	We can write $\varphi^r(\rho_{k_0-r})$ as a product of $\rho_{k_0}'$ and a subword of $\pi^{-1}\varphi^r(\pi)$.
	Hence, by Proposition~\ref{prop:fbc subword length},
	$$\abs{\varphi^r(\rho_{k_0-r})}_H
	\ \leq \ \abs{\rho_{k_0}'}_H + (2m+1) \abs{\pi^{-1}\varphi^r(\pi)}_H.
	$$
These displayed inequalities  combine to give an upper bound on $k_0$   implying the condition on $\Lambda_1$.  

	\case{2}  $\pi = a_i u a_i^{-1}$ for a reduced word $u$ of rank less than $i$. 
	
	Let $\pi_0 = a_i u$.  The only $a_i^{-1}$ in $\pi$ is at the end; ditto in $\varphi^r(\pi)$.  So if the common prefix $L$ of $\pi$ and $\varphi^r(\pi)$ is the whole of $\pi$, then $L=\varphi^r(\pi)$, also,  but that cannot be: by Lemma~\ref{lem:fixed pieces} $\pi$ is  not fixed by $\varphi$ and then by Proposition~\ref{prop:lower bound distortion} it is not fixed by  $\varphi^r$ (as  $r \neq 0$).  So $L$ is, in fact, a   prefix of  $\pi_0$.  

As $\varphi^r(\pi)$ is the free reduction of $\varphi^r(\pi_0)\varphi^r(a_i^{-1})$, the word $L$ is the concatenation $\Lambda_1\Lambda_2$ of a prefix $\Lambda_1$ of $\varphi^r(\pi_0)$ with a subword $\Lambda_2$ of $\varphi^r(a_i^{-1})$.
 But then $\Lambda_1$ is a common prefix of $\pi_0$ and $\varphi^r(\pi_0)$ (though it may not be the full common prefix)
so we deduce from Case~1 that $\Lambda_1$ is a prefix of $\varphi^k(a_i)$, 
where $\abs{k}$ is bounded by a constant times $r + \abs{\pi_0^{-1}\varphi^r(\pi_0)}_H$.
Since $\pi_0^{-1}\varphi^r(\pi_0) = a_i^{-1} \pi^{-1}\varphi^r(\pi) a_i$, we have $\abs{\pi_0^{-1}\varphi^r(\pi_0)}_H \leq \abs{\pi^{-1}\varphi^r(\pi)}_H + 2$ and the required bound on $\abs{k}$ follows.
\end{proof}

\begin{cor}\label{cor:fbc shared prefix of words}
	For all  $1\leq i \leq m$, there exists $A_i >0$ with the following property. For all freely reduced
   $w \in F$   of rank $i$ and all $r \in \Z$,  
  there exists  a freely reduced word $w_0 \in F$ of rank at most $i$ such that the following hold. 
  \begin{enumerate}\renewcommand{\theenumi}{$\mathcal{P}$\arabic{enumi}}
  	\item \label{shared prefix change w} 
  		If the free reduction of $w^{-1}\varphi^r(w)$ is $\alpha \beta$, with $\alpha$ a prefix of $w^{-1}$ and $\beta$ a suffix of $\varphi^r(w)$,
  		then the free reduction of $w_0^{-1}\varphi^r(w_0)$ is $\alpha \beta$, and $\alpha$ a prefix of $w_0^{-1}$ and $\beta$ a suffix of $\varphi^r(w_0)$.
  \item \label{shared prefix form} The longest common prefix $P$  of $w_0$ and $\varphi^r(w_0)$  
  has the form  $P =P_1 P_3 \cdots P_i$ where 
  \begin{itemize} 
  	\item $P_1$ is a  prefix  of $\varphi^k(a_t)$ for some $t\leq i$ and some $k\in \Z$ satisfying $\abs{k} \leq  A_i \left(\abs{\alpha \beta }_H + \abs{r}\right)$,  and 
  	\item  $P_j$ is a subword of $\varphi^r(a_j^{-1})$ for $j=3, \ldots, i$.
  \end{itemize}
  \end{enumerate}
\end{cor}

\begin{proof}
	If  $w = \varphi^r(w)$, then we can take $\alpha$, $\beta$, and $w_0$ to be the empty word. 
	So assume   $w \neq \varphi^r(w)$. In particular, $r\ne 0$.
	
	The statement for $r<0$ will follow  from that for $r>0$ since we could instead consider the common prefix of $\bar{w} = \varphi^r(w)$ and $\varphi^{-r}(\bar{w})$, which of course equals  $P$.  So assume $r>0$.

	We will  induct on  $i = \rank(w)$.
	The case   $i=1$ is elementary:  $w = \varphi^r(w)$ and the result holds as we just explained.

	Assume $ i  >1$.  Let $w = \pi_1\cdots \pi_p$ be the rank-$i$ decomposition of $w$ into pieces.
	Our first step is to reduce the problem to a question concerning a single piece.
	Take $k$ minimal so that the longest common prefix of  $w$ and $\varphi(w)$ is a    subword of $\pi_1 \cdots \pi_{k}$.
	Let $\pi := \pi_k$. 
	It follows from Lemma~\ref{lem:pieces exercise} that $\pi_1$, \ldots,  $\pi_{k-1}$  are   fixed by $\varphi^r$.
	So \eqref{shared prefix change w}  amounts to: 
	\begin{enumerate}\renewcommand{\theenumi}{$\mathcal{P}$\arabic{enumi}$^{\prime}$}
	\item	\label{v cut down}
	If the free reduction of $w^{-1}\varphi^r(w)$ is $\alpha \beta$, with $\alpha$ a prefix of $w^{-1}$ and $\beta$ a suffix of $\varphi^r(w)$,
	then the free reduction of $ (\pi_k \cdots \pi_p)^{-1}\varphi^r(\pi_k \cdots \pi_p)$ is $\alpha \beta$, with $\alpha$ a prefix of $( \pi_k \cdots \pi_p)^{-1}\varphi^r(\pi_k \cdots \pi_p)$ and $\beta'$ a suffix of $\varphi^r(\pi)$.
	\end{enumerate}
	
	We will find a word $\pi_0$ such that the longest common prefix $P$ of $\pi_0$ and $\varphi^r(\pi_0)$ satisfies the conditions for  \eqref{shared prefix form}, and $\pi_0$ satisfies:
	\begin{enumerate}\renewcommand{\theenumi}{$\mathcal{P}$\arabic{enumi}$^{\prime\prime}$}
	\item \label{pi cut down}
		If the free reduction of $\pi^{-1}\varphi^r(\pi)$ is $\alpha' \beta'$, with $\alpha'$ a prefix of $\pi^{-1}$ and $\beta'$ a suffix of $\varphi^r(\pi)$,
	then the free reduction of $\pi_0^{-1}\varphi^r(\pi_0)$ is $\alpha' \beta'$, with $\alpha'$ a prefix of $\pi_0^{-1}$ and $\beta'$ a suffix of $\varphi^r(\pi_0)$.
	\end{enumerate}
	Then, setting $w_0 = \pi_0\pi_{k+1}\ldots \pi_p$, we will get $\alpha \beta = w_0^{-1}\varphi^r(w_0)$ in $F$ by \eqref{v cut down},  with $\alpha$ a prefix of $w_0^{-1}$ and $\beta$ a suffix of $\varphi^r(w_0)$, satisfying \eqref{shared prefix change w}.
	We claim that $P$ being the longest common prefix of $\pi_0$ and $\varphi^r(\pi_0)$ implies it is  also  the longest common prefix of $w_0$ and $\varphi^r(w_0)$. 
	Indeed, if $w_0 = \pi_0$ (as words), this is immediate.  
	And if $w_0 \ne \pi_0$, then either the prefixes are as claimed, or $\pi_0 = \varphi^r(\pi_0)$, implying $\pi^{-1}\varphi^r(\pi)$ is trivial, by \eqref{pi cut down}, contradicting our choice of $k$.
	
	Now we turn to finding this $\pi_0$.
		
	Assume that $i=2$, which is an exceptional case.
	By the minimality of $k$, we cannot have  $\pi$ equal to $a_1^q$ or $a_2a_1^qa_2^{-1}$ for some $q \in \Z$,   for that would imply that $\pi =  \varphi^r(\pi)$. 
	The remaining possibilities are that  $\pi$ is $a_2a_1^q$ or $a_1^q a_2^{-1}$ for some $q \in \Z$.  Take $\pi_0  = a_2$ or $\pi_0 = a_2^{-1}$, 
	respectively in these two cases.  
	 In the former case, $\alpha'$ is the empty word, and $\beta' = a_1^r$. In the latter, $\alpha' = a_2$ and $\beta' = a_1^{-r}a_2^{-1}$. Both satisfy \eqref{pi cut down}.
	The longest common prefix of $\pi_0$ and $\varphi^r(\pi_0)$, and hence also of $w_0$ and $\varphi^r(w_0)$, is then either $a_2$ or the empty word. Taking    $P_1$ to be $a_2$ or the empty word, accordingly, and $P_3, \ldots, P_i$ all the empty word  gives us the required form \eqref{shared prefix form}.

	Suppose  $i\geq 3$.  How we proceed depends on the  type and rank   of the piece $\pi$.
	
	\case{1}  $j := \rank(\pi) < i$.  
	
	Apply the induction hypothesis to get a freely reduced word $\pi_0$ of rank at most $j$  satisfying \eqref{pi cut down}
	and such that the longest common prefix of $\pi_0$ and $\varphi^r(\pi_0)$ has the form $P = P_1P_3\cdots P_j$, where $P_3,\ldots, P_j$ are each subwords of $\varphi^r(a_j^{-1})$, and $P_1$ a prefix of $\varphi^k(a_t)$, for some $t\leq j$ and some $\abs{k} \leq A_j(\abs{\pi^{-1}\varphi^r(\pi)}_H + \abs{r})$.
	Since $\pi^{-1}\varphi^r(\pi)$ is a subword of $\alpha\beta$, taking $A_i$ large enough that $A_i \geq (2m+1)A_j$ will mean, by Proposition~\ref{prop:fbc subword length}, that $P_1$ satisfies the requirements stated in this corollary.
	Finally, we take $P_{j+1},\ldots,P_i$ to all be the empty word.

	\case{2} The first letter of $\pi$ is $a_i$.	
	
   In this case we take $\pi_0 = \pi$, so \eqref{pi cut down} trivially holds.
	Lemma~\ref{lem:fbc shared prefix} gives us 
	the structure of $P$ as required, with $P_1=\Lambda_1$, with $P_3,\ldots,P_{i-1}$ being empty words, and with $P_i = \Lambda_2$. 
	We just note that the power $k$ in $P_1$ satisfies the required bound by taking $A_i$ large enough so that $A_i \geq (2m+1)B_i$, where $B_i$ is the constant from Lemma~\ref{lem:fbc shared prefix} (as in Case 1, this is because $\pi^{-1}\varphi^r(\pi)$ is a subword of $\alpha\beta$, and we can apply Proposition~\ref{prop:fbc subword length}).
	
	\case{3} $\pi = ua_i^{-1}$ with $j := \rank(u) <i$.

	By the inductive hypothesis, there is a word $u_0 \in F$ of rank at most $j$ such that 
	\begin{enumerate}\renewcommand{\theenumi}{$\mathcal{P}$\arabic{enumi}$^{\prime\prime\prime}$}
\item \label{eq: u to u_0 prefix}
	If the free reduction of $u^{-1}\varphi^r(u)$ is $\gamma \delta$, with $\gamma$ a prefix of $u^{-1}$ and $\delta$ a suffix of $\varphi^r(u)$,
	then the free reduction of $u_0^{-1}\varphi^r(u_0)$ is $\gamma \delta$, with $\gamma$ a prefix of $u_0^{-1}$ and $\delta$ a suffix of $\varphi^r(u_0)$,
	\end{enumerate} 
	and the maximal common prefix of $u_0$ and $\varphi^r(u_0)$ is of the form   
	$P_0 = P_1P_3\cdots P_j$, with $P_1,P_3\ldots,P_j$ as per \eqref{shared prefix form}.
	In particular, $P_1$ is a prefix of $\varphi^k(a_t)$, for some $t\leq j$ and some $k$ satisfying 
	\begin{equation}\label{eq:k bound case 3}
	\abs{k}  \ \leq \  A_j\left(\abs{u_0^{-1}\varphi^r(u_0)}_H + \abs{r}\right).
	\end{equation}
	
	Let $\pi_0 = u_0a_i^{-1}$.
	Since $\varphi^r(\pi_0)$ is the free reduction of $\varphi^r(u_0)\varphi^r(a_i^{-1})$, the common prefix $P$ of $\pi_0$ and $\varphi^r(\pi_0)$ is a concatenation of a common prefix of $u_0$ and $\varphi^r(u_0)$ (hence a prefix of $P_0$), and a (possibly empty) subword $P_i$ of $\varphi^r(a_i)^{-1}$.
	We may therefore write $P = P_1P_3\cdots P'_{j'}P_i$ for some $j' \leq j$, where $P'_{j'}$ is a prefix of $P_{j'}$.
	This fits with the structure in \eqref{shared prefix form}, although we still need to determine the bound on $k$, which we will do below, while we verify \eqref{pi cut down} holds.

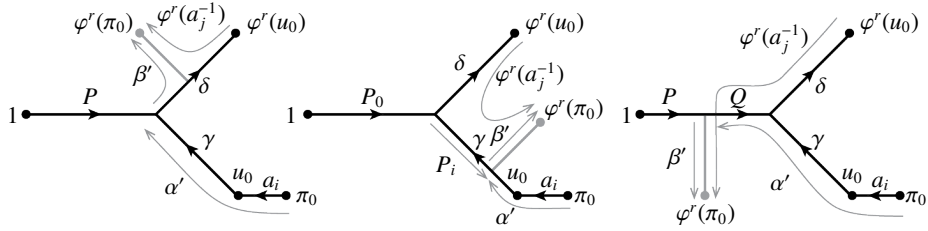
\begin{figure}[ht]
\begin{overpic}[
scale=1.0,unit=1mm]{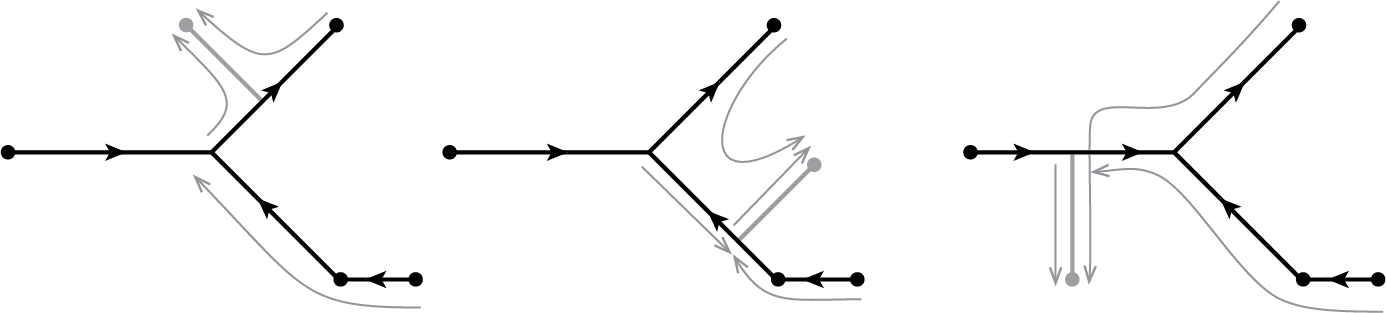}
 \put(-2,12.5){\small{$1$}}         
 \put(36.5,2){\small{$\pi_0$}}         
 \put(30,25){\small{$\varphi^r(u_0)$}}         
 \put(28,5){\small{$u_0$}}         
 \put(32,4){\small{$a_i$}}         
  \put(19,3){\small{$\alpha'$}}         
 \put(15,18){\small{$\beta'$}}         
  \put(24,9.5){\small{$\gamma$}}         
  \put(23.5,16.5){\small{$\delta$}}         
 \put(7,25){\small{$\varphi^r(\pi_0)$}}         
 \put(18,26){\small{$\varphi^r(a_j^{-1})$}}         
\put(8,15){\small{$P$}}         
 \put(35,12.5){\small{$1$}}         
 \put(73.5,2){\small{$\pi_0$}}         
 \put(67,25){\small{$\varphi^r(u_0)$}}         
 \put(64.5,5){\small{$u_0$}}         
 \put(69,4){\small{$a_i$}}         
  \put(63,-1){\small{$\alpha'$}}         
 \put(55,6){\small{$P_i$}}         
 \put(62.5,10){\small{$\beta'$}}         
  \put(60,9.5){\small{$\gamma$}}         
  \put(57.5,19.5){\small{$\delta$}}         
 \put(69.5,13.5){\small{$\varphi^r(\pi_0)$}}         
 \put(63,18){\small{$\varphi^r(a_j^{-1})$}}         
\put(45,15){\small{$P_0$}}         
 \put(79.6,12.5){\small{$1$}}         
 \put(117.5,2){\small{$\pi_0$}}         
 \put(111.5,25){\small{$\varphi^r(u_0)$}}         
 \put(109,5){\small{$u_0$}}         
 \put(113,4){\small{$a_i$}}         
  \put(99.5,3.5){\small{$\alpha'$}}         
 \put(86,7){\small{$\beta'$}}         
  \put(105,9.5){\small{$\gamma$}}         
  \put(105.5,16.5){\small{$\delta$}}         
 \put(87,-0.5){\small{$\varphi^r(\pi_0)$}}         
 \put(95.5,23){\small{$\varphi^r(a_j^{-1})$}}         
 \put(94,15){\small{$Q$}}         
 \put(85,15){\small{$P$}}  
 \end{overpic}
 \caption{Three cases of cancellation between $\varphi^r(a_i^{-1})$ and other words}
  \label{fig:fbc2019}
\end{figure}

	We seek $\alpha'$ and $\beta'$ satisfying \eqref{pi cut down}, so we need to understand the free reduction of $\pi_0^{-1} \varphi^r(\pi_0)$.
	We have 
	$$\pi_0^{-1} \varphi^r(\pi_0) 
	 \ = \  a_i u_0^{-1} \varphi^r(u_0) \varphi^r(a_i)^{-1}
	\ = \  a_i \gamma \delta \varphi^r(a_i)^{-1}.$$	
	Cancellation in $a_i \gamma \delta \varphi^r(a_i)^{-1}$ can only occur where $\delta$ abuts $\varphi^r(a_i)^{-1}$. 
	One of three things can occur. 
	Either 
	\begin{enumerate}\renewcommand{\theenumi}{$\mathcal{C}$\arabic{enumi}}
		\item \label{prefix1} $\varphi^r(a_i)^{-1}$ does not cancel into either $\gamma$ or $P_0$ (the left diagram of Figure~\ref{fig:fbc2019}),
		\item \label{prefix2} $\varphi^r(a_i)^{-1}$ completely cancels with $\delta$ and continues cancelling into $\gamma$ (the middle   diagram in Figure~\ref{fig:fbc2019}), or
		\item \label{prefix3} $\varphi^r(a_i)^{-1}$ completely cancels with $\delta$ and continues cancelling into $P_0$ (the right diagram in Figure~\ref{fig:fbc2019}).
	\end{enumerate}

	In \eqref{prefix1} we can take $\alpha' = a_i\gamma$ and $\beta'$ to be the free reduction of $\delta \varphi^r(a_i)^{-1}$.  Then it is straightforward to see that \eqref{pi cut down} holds.
	
	For \eqref{prefix2}, we take $P_i$ as above: it is the subword of $\varphi^r(a_i)^{-1}$ that cancels into $\gamma$.
	Then, setting $\alpha'$ to be the free reduction of $a_i \gamma P_i$  and $\beta'$ to be the free reduction of $P_i^{-1} \delta \varphi^r(a_i)^{-1}$, we can check that \eqref{pi cut down} is satisfied.
	Take $P=P_0P_i$.
	
	Finally, for \eqref{prefix3}, we let $Q$ be the suffix of $P_0$ that cancels into $\varphi^r(a_i)^{-1}$, and $P$ be the prefix of $P_0$ so that $P_0 = PQ$ as words. (Note in this case $P_i$ is empty.)
	Then $\alpha ' = a_i \gamma Q^{-1}$ and $\beta'$ equal to the free reduction of $Q \delta \varphi^r(a_i)^{-1}$ satisfy \eqref{pi cut down}.
	
	In each case,  the free reduction of $u_0^{-1} \varphi^r(u_0)$ is a product of a subword of the free reduction of $\pi_0^{-1}\varphi^r(\pi_0)$ with a subword of $\varphi^r(a_i)^{\pm 1}$.
	By \eqref{eq: u to u_0 prefix}, $\pi_0^{-1}\varphi^r(\pi_0)$ is equal to $\pi^{-1}\varphi^r(\pi)$, which is a subword of $\alpha\beta$.
	Hence by Proposition~\ref{prop:fbc subword length}, $\abs{ u_0^{-1} \varphi^r(u_0) }_H \leq (2m+1)\left( \abs{ \alpha\beta }_H + 2\abs{r} + 1\right)$.  In each case $P$ has the form required for  \eqref{shared prefix form}, and
	this bound, together with \eqref{eq:k bound case 3} and  increasing the value as $A_i$ if necessary, gives the required bound on $\abs{k}$. 
\end{proof}

\section{The inductive structure of $\mathcal{H}$-twisted conjugacy in $F$}
\label{H-twisted section1}

We explained in Section~\ref{reductions} that the conjugacy problem in $H$ amounts to a twisted conjugacy problem in $F$ which can take one of three forms.   The most involved of the three is what we refer to as the $\H$-twisted conjugacy problem.  We will show here that when this problem has a solution, it has a solution of one of a number of  particular forms.  The large number of possibilities for this form  leads to the following proposition  having a somewhat involved statement.  But all the subwords  are described in terms of the `constants'  $p, u_0,v_0,u_1,v_1$ or inductively in terms of a word $\hat{x}$ which is a solution to a lower rank instance of the same problem.  This will allow us to estimate the  lengths of solutions.  Those estimates will feed into upper bounds on the conjugator length of $H$.    Also this proposition will mean that solutions to the   $\H$-twisted conjugacy problem can found by searching though a polynomially sized family of possibilities.  This will feed into  polynomial time solutions to the conjugacy and conjugacy-search problems in $H$.

\begin{prop}\label{conjugator-like problem in F}
Suppose $i \in \set{1, \ldots, m}$, $p >0$ are  integers  and  $u_0,v_0,u_1,v_1,x$ are  reduced words on $a_1^{\pm 1}, \ldots, a_m^{\pm 1}$.  Suppose $x$ is  non-empty word  and has rank $i$.  Suppose  the concatenations $u_0  x  v_0^{-1}$ and $u_1^{-1} x v_1$ are reduced words and satisfy  	
		\begin{equation}   \label{twisted with x}
	    \varphi^{-p}(u_0  x  v_0^{-1})  \ = \  u_1^{-1} x v_1  \  \text{  in }  \   F. 
	  	\end{equation}
Then there exists a word $X$ 	on $a_1^{\pm 1}, \ldots, a_m^{\pm 1}$  	
that satisfies 	
\begin{equation}   \label{twisted with X}
	    \varphi^{-p}(u_0  X  v_0^{-1})  \ = \ u_1^{-1}  X v_1  \  \text{  in }  \   F 
\end{equation}
and takes the following form.
If $i = 1$, then 
\begin{enumerate}\renewcommand{\theenumi}{$\mathcal{X}$\arabic{enumi}} 
	\item \label{form of x 0} $X    =    U_1 U_2 U_3 V$ as words
\end{enumerate}
for some subwords $U_1$, $U_2$, and $U_3$  of
$(u_1  \varphi^{-p}(u_0))^{\pm 1}$ and some suffix $V$  of $v_1 \varphi^{-p}(v_0)$.

There exists a constant $C>0$ such that if $i>1$, then either
\begin{enumerate}\renewcommand{\theenumi}{$\mathcal{X}$\arabic{enumi}} 
	\setcounter{enumi}{1}
	\item \label{form of x 1} $X=x$ is a subword {$L$} of $\varphi^{-p}(u_0)$ or {$R$} of $\varphi^{-p}(v_0^{-1})$,
	\item \label{form of x 2} $X = L \ S \ M \  P \ R$ as words, or
	\item \label{form of x 3} $X   = x = L \ \hat{x} \ R$ as words,
\end{enumerate}
where 
\begin{itemize}
		\item $L$ is a subword of $\varphi^{-p}(u_0)$,   
		\item $S$ is either
		\begin{itemize}
			\item $S = S_i \cdots S_3\ S_1$ where 
			\begin{itemize}
				\item $S_1^{-1}$ is a prefix of $\varphi^{k}(a_t)$ for some $t\leq i$ and $\abs{k} \leq  C \left( \abs{u_0}_F + \abs{u_1}_F +  p \right)$,  
				\item $S_j$ is a subword of $\varphi^{-p}(a_j)$ for $j=3,\ldots, i$,
			\end{itemize} 
			\item a subword of $\hat{S}$ where $\varphi^{-p}(\hat{S})$ is a subword of $(u_1\varphi^{-p}(u_0))^{-1}$, or 
			\item a subword of $\varphi^{-p}(\hat{S})$ where $\hat{S}$ is a subword of $\varphi^{-p}(u_0)$, 
		\end{itemize}
		\item $M = M_1 M_2$ or $M_2^{-1}M_1^{-1}$, where
		\begin{itemize}
			\item  $M_1 = \pi \varphi^p(\pi) \cdots \varphi^{p(q-1)}(\pi)$,  where 
			\begin{itemize}
				\item $qp \leq C (\abs{u_0}_H  + \abs{u_1}_H  + \abs{v_0}_H  + \abs{v_1}_H  + p)$, and
				\item $\varphi^{-p}(\pi)$ is  a concatenation of a subword of $(u_1\varphi^{-p}(u_0))^{-1}$ with $S$, 
					or of $P$ with a subword of $(v_1 \varphi^{-p}(v_0))^{-1}$, 
			\end{itemize}
			\item  $\varphi^{p}(M_2)$ is a subword of $\varphi^{-p}(u_0)$ or  $\varphi^{-p}(v_0^{-1})$, 
		\end{itemize}
		\item $P$ is either
		\begin{itemize}			
			\item $P = P_1\ P_3\cdots P_i$, where
				\begin{itemize}
					\item $P_1$ is a prefix of $\varphi^{k'}(a_t)$ for some $t\leq i$ and $\abs{k'} \leq C \left(\abs{v_0}_F +  \abs{v_1}_F +  p\right)$,
					\item $P_j$ is a subword of $\varphi^{-p}(a_j^{-1})$ for $j=3,\ldots, i$, or
				\end{itemize}
			\item a subword of $\varphi^{-p}(\hat{P})$ where $\hat{P}$ is a subword of $\varphi^{-p}(v_0^{-1})$, 
			\item a subword of $\hat{P}$ where $\varphi^{-p}(\hat{P})$ is a subword of $v_1\varphi^{-p}(v_0)$, 
		\end{itemize}	
		\item $R$ is a subword of $\varphi^{-p}(v_0^{-1})$,
		\item $\hat{x}$ has rank $j<i$ and satisfies $ \varphi^{-p}(\hat{u}_0  \hat{x}  \hat{v}_0^{-1})  \ = \ \hat{u}_1^{-1}  \hat{x} \hat{v}_1$ where  
		\begin{itemize}			
			\item  $\hat{u}_0\hat{u}_1$ is reduced and is a subword of $\varphi^{-p}(u_0)$, 
			\item  $\hat{v}_0\hat{v}_1$ is reduced and is a subword of  $\varphi^{-p}(v_0)$. 
		\end{itemize}
	\end{itemize}
 \end{prop}

A curious feature of Proposition~\ref{conjugator-like problem in F} is that if $u,  v \in H$ satisfy $uw=wv$ in $H$ for some $w \in F$ (not just in $H$), and    $w=u_0xv_0^{-1}$ as per the proposition,  then  $u_0Xv_0^{-1}$, which will also be in $F$, is another conjugator  and  also has this `nice' structure.

In general, this structure leads to a quadratic upper bound on the length (see Lemma~\ref{lem:chunk lenght bound} below).  To improve it to  a linear upper bound we need to replace the word $M_1$ appearing in \eqref{form of x 2} with $s^{qp}$ (see Lemma~\ref{lem:swap for a linear conjugator} below). 
We therefore swap $X$ for a word $\hat{X}$, which unlike $X$ may represent an element of $H \ssm F$.  Equations \eqref{twisted with x} and \eqref{twisted with X} may therefore not make sense for $\hat{X}$. 
For the iteration through case \eqref{form of x 3}, then, we will instead use:

\begin{lemma}\label{lem:annoying iteration}
	With the notation from Proposition~\ref{conjugator-like problem in F}, if $X$ has form \eqref{form of x 3}, 
	and if $\hat{X} \in H$  satisfies 
	$$s^p \ \hat{u}_0\  \hat{X}\  \hat{v}_0^{-1} \ s^{-p}  \ = \ \hat{u}_1^{-1} \ \hat{X} \ \hat{v}_1  \ \ \text{ in } H$$
	then 
	\begin{equation} \label{LR in w} 
	s^p \ u_0 \ L\ \hat{X}\ R \ v_0^{-1} s^{-p} \ = \ u_1^{-1} \ L \ \hat{X}\ R \ v_1  \  \  \text{ in } H.
	\end{equation}
	If   $\tilde{w} =  u_0 \ L\ \hat{X}\ R \ v_0^{-1}$, then \eqref{LR in w} amounts to $uw=wv$ in $H$, where $w = \tilde{w} s^r$, $u = \tilde{u} s^p$, $v = \tilde{v} s^p$, $\tilde{u} = u_0 u_1$, and $\varphi^{-r}(\tilde{v}) = v_0 v_1$.
\end{lemma}

We will prove this lemma after proving  Proposition~\ref{conjugator-like problem in F}.  Before we prove either, here is a lemma which is straight-forward, but which we highlight as we will call on it to remove a subword from $x$.  

\begin{lemma} \label{easy insert}
Suppose  $x=x_0 x_1$,  $\varphi^{-p}(u_0  x_0)  =  u_1^{-1} x_0$ ,  $\varphi^{-p}(x_1 v_0^{-1}  )  =  x_1v_1 $,   and $\varphi^{-p}(y) = y$ in  $F$.  Then $\varphi^{-p}(u_0  x_0 y x_1  v_0^{-1})   =   u_1^{-1} x_0 y x_1 v_1$  in $F$. 
\end{lemma}

\begin{proof}[Proof of Proposition~\ref{conjugator-like problem in F}]
We begin with the  case $i=1$.  As $\rank(x) = 1$, $\varphi(x)=x$, and so  \eqref{twisted with x} rearranges to the conjugacy relation
	\begin{equation} \label{x as a conjugator}
u_1  \varphi^{-p}(u_0) x \ = \  x v_1  \varphi^{-p}(v_0) \  \text{  in }  \   F.
	  	\end{equation}
     Therefore, by Lemma~\ref{CP in F lemma}, there is \emph{some} $x_0 \in F$ which satisfies \eqref{x as a conjugator}  in place of $x$   (but may fail to satisfy  \eqref{twisted with x}   since $\varphi$ need not  fix  $x_0$) and  is the concatenation of some prefix  of $(u_1 \varphi^{-p}(u_0))^{ -1}$ with some suffix  of $v_1 \varphi^{-p}(v_0)$.

If $u_1  \varphi^{-p}(u_0) = 1$, then  \eqref{twisted with X}   holds with $X$ the empty word.  Assume, then, that $u_1  \varphi^{-p}(u_0) \neq 1$.  Since both $x$ and $x_0$ conjugate $u_1  \varphi^{-p}(u_0)$ to $v_1  \varphi^{-p}(v_0)$ in $F$, we have $x = \sigma^l x_0$ in $F$ for some integer $l$ and some reduced word $\sigma$ some power of which freely equals   $u_1  \varphi^{-p}(u_0)$.      
 If  $\rank(\sigma)=1$, then $\rank( x_0)= 1$ also and so  $\varphi(x_0) = x_0$ and  \eqref{twisted with X}   holds with $X = x_0$.     
  If, on the other hand, the $\rank(\sigma) \geq 2$, then  take $X=x$.  Then    \eqref{twisted with X} is \eqref{twisted with x} and so holds.  And,    as $\rank(x)=1$,   all of   $\sigma^l$ apart from some prefix $\sigma'$ of $\sigma$   (if $l>0$)  or $\sigma^{-1}$ (if $l<0$) must cancel  into  $x_0$ in  $\sigma^l x_0$, and so  $X$ is the concatenation of a prefix $\sigma'$ of $\sigma$ or  $\sigma^{-1}$ with a suffix of $x_0$.    
  
  Let   $\sigma_0$ be the maximal suffix of $\sigma$ such that  $\sigma_0^{-1}$ is  prefix of $\sigma$.  Then  there is a subword $\sigma_1$ of $\sigma$  such that  for all $c \in \Z$, as words  $\sigma^c = \sigma_0^{-1} \sigma^c_1 \sigma_0$.  So, as some power of $\sigma$ freely equals   $u_1  \varphi^{-p}(u_0)$, we have that  $\sigma_0^{-1} \sigma_1$ freely equals a prefix of  $(u_1  \varphi^{-p}(u_0))^{\pm 1}$ and  $\sigma_0$  freely equals a suffix.  So $\sigma$ freely equals the concatenation of two subwords of   $(u_1  \varphi^{-p}(u_0))^{\pm 1}$, and the same is true of $\sigma'$.

So in each case $X$ is   the   concatenation of three subwords of   $(u_1  \varphi^{-p}(u_0))^{\pm 1}$ with a suffix  of $v_1 \varphi^{-p}(v_0)$, completing the proof in the case $i=1$.  
  
Now assume $2\le i \le m$.
By \eqref{twisted with x}, $\varphi^{-p}(u_0)\varphi^{-p}(x)\varphi^{-p}(v_0^{-1}) = u_1^{-1}xv_1$ in $F$. The right-hand side is reduced, but there may be cancellation on the left at the start and end of $\varphi^{-p}(x)$.
So, after free reduction on the left, a (perhaps empty) subword of $\varphi^{-p}(x)$ remains. This subword will also be a subword of $u_1^{-1}xv_1$. Define $\Pi$ to be its overlap with $x$.

If $\Pi$ is the empty word, then  $x$ is a subword of either  $\varphi^{-p}(u_0)$ or $\varphi^{-p}(v_1^{-1})$, and so  $X=x$ has the form of $\eqref{form of x 1}$. So for the remainder of the proof we assume that $\Pi$ is nonempty.

As $\Pi$ is a subword of both $x$ and $\varphi^{-p}(x)$, we can define $L$, $R$, $Y$ and $Z$ so that
\begin{equation} \label{LRYZ defined}
x = L\ \Pi \ R  \qquad
\text{ and }  \qquad \varphi^{-p}(x) = Y \ \Pi \  Z \qquad \text{as words.}
\end{equation}
We have
$$u_1^{-1}\ L\ \Pi \ R \ v_1 \  =  \ \varphi^{-p}(u_0)\ Y \ \Pi \ Z\ \varphi^{-p}(v_0^{-1}) \qquad \text{in } F$$ 
by \eqref{twisted with x}.  
We claim that
\begin{enumerate}\renewcommand{\theenumi}{$\mathcal{E}$\arabic{enumi}}
	\item \label{structure of L} $L$ is a subword of $\varphi^{-p}(u_0)$,
	\item \label{structure of R} $R^{-1}$ is a subword of $\varphi^{-p}(v_0)$,
	\item \label{structure of Y} $Y^{-1}$ is a suffix of the freely reduced form of $u_1\varphi^{-p}(u_0)$, and
	\item \label{structure of Z} $Z$ is a suffix of  the freely reduced form of  $v_1\varphi^{-p}(v_0)$.
\end{enumerate}

By definition of $\Pi$,  
\begin{equation} \label{what cancellation}
\varphi^{-p}(u_0)  Y  \ \text{ freely reduces to }  \  u_1^{-1}\ L.
\end{equation}
Also, $\varphi^{-p}(u_0)$ and $Y$ are freely reduced words, so if $Y$ fully cancels into $\varphi^{-p}(u_0)$ on free reduction of $\varphi^{-p}(u_0)  Y$, then $Y^{-1}$ is a suffix of  $\varphi^{-p}(u_0)$ and  \eqref{what cancellation} gives us  \eqref{structure of L} and  \eqref{structure of Y}.
If, on the other hand, a non-empty suffix of $Y$ survives free reduction of $\varphi^{-p}(u_0)\ Y$, then $L$ is the empty word:   the last letter of $L$, were there one, would have to have been part of $\Pi$   since it would also have to have been the last letter of $Y$.  So \eqref{structure of L} trivially holds and   \eqref{structure of Y} again follows from \eqref{what cancellation}.

To complete the proof, when $\Pi$ is nonempty we need to explain how to replace $\Pi$ with some $\hat{\Pi}$ so that $X= L\hat{\Pi} R$ is of the required form.

Let $\pi_1 \cdots \pi_k$ be the rank-$i$ piece decomposition of $x$.  
By Lemma~\ref{lem:pieces exercise},   $\varphi^{-p}(\pi_1) \cdots 	\varphi^{-p}(\pi_k) $ is the rank-$i$ decomposition of $\varphi^{-p}(x)$ into pieces and, in particular, is reduced.  

Since $\Pi$ is a subword of $x$,  there are integers $a,b$ such that $x$ is a subword of $\pi_a \cdots \pi_b$. We choose $a$ and $b$ so that either
	\begin{enumerate}[I.]\renewcommand{\theenumi}{\Roman{enumi}}
	\item \label{case2}  $a<b$ and $\Pi = \pi_a' \pi_{a+1} \cdots \pi_{b-1}\pi_b'$, for some (perhaps empty) suffix $\pi'_a$  of $\pi_a$ and (perhaps empty) prefix $\pi'_b$ of $\pi_b$,  where $\pi_a' \ne \pi_a$ if $1 < a$ and $\pi_b'\ne \pi_b$ if $b<k$, or 
	\item \label{case1}  $a=b$ and $\Pi = \pi_a'$ for a nonempty subword $\pi_a'$ of $\pi_a$.  
	\end{enumerate} 
	As $\Pi$ is a subword of $\varphi^{-p}(x)$,  its rank--$i$  pieces   line  up with those in $\varphi^{-p}(\pi_1) \cdots 	\varphi^{-p}(\pi_k)$.  This tells us
	\begin{enumerate}\renewcommand{\theenumi}{$\mathcal{S}$\arabic{enumi}}
		\item \label{shift setup}  
		$\Pi$ is a subword of  $\varphi^{-p}(\pi_{a+e}) \cdots \varphi^{-p}(\pi_{b+e})$ for some $e \in \Z$.
		\newcounter{shiftcount}\setcounter{shiftcount}{\value{enumi}}
	\end{enumerate}  
	The value of $e$ will be important.  

  From \eqref{LRYZ defined} and \eqref{structure of L}--\eqref{structure of Z} we  make   the following observations.  

 	\begin{enumerate}\renewcommand{\theenumi}{$\mathcal{S}$\arabic{enumi}}
	\setcounter{enumi}{\theshiftcount}
	\item \label{shift outside1} 
	If $c<a$, then $\pi_c$ is a subword of $\varphi^{-p}(u_0)$.
	\item \label{shift outside4} 
	If $c>b$, then $\pi_c^{-1}$ is a subword of $\varphi^{-p}(v_0)$.
	\item \label{shift outside2} 
	If $c<a+e$, then $\varphi^{-p}(\pi_c)^{-1}$ is a subword of the freely reduced form of  $u_1\varphi^{-p}(u_0)$. 
	\item \label{shift outside3} 
	If $c>b+e$, then $\varphi^{-p}(\pi_c)$ is a subword of  the freely reduced form $v_1\varphi^{-p}(v_0)$.
	\setcounter{shiftcount}{\value{enumi}}
\end{enumerate}

\case{\ref{case2}} $\Pi = \pi'_a \pi_{a+1} \cdots \pi_{b-1} \pi'_b$.

We make the following further observations that apply in this case. 
They follow by considering $\Pi$ simultaneously as a subword of $x$, which has piece decomposition $\pi_1 \cdots \pi_k$, and of $\varphi^{-p}(x)$, with piece decomposition $\varphi^{-p}(\pi_1) \cdots \varphi^{-p}(\pi_k)$.
As mentioned, the breaks between pieces in $\Pi$ must line up in the two words $x$ and $\varphi^{-p}(x)$.

\begin{enumerate}\renewcommand{\theenumi}{$\mathcal{S}$\arabic{enumi}}
	\setcounter{enumi}{\value{shiftcount}}
\item \label{shift initial} $\pi'_a$ is a suffix of $\varphi^{-p}(\pi_{a+e})$ (as well as of $\pi_a$),
\item \label{shift} $\pi_i = \varphi^{-p}(\pi_{i+e})$ for $a+1 \leq i  \leq b-1$,   
\item \label{shift final} $\pi'_b$ is a prefix of $\varphi^{-p}(\pi_{b+e})$ (as well as of $\pi_b$).
\end{enumerate}

We will show that the proposition is satisfied with $X$ as per \eqref{form of x 2}.
We will divide into three subcases according to the value of $e$. Having $e$  positive is similar to $e$ negative. Indeed, taking inverses of both sides of equation \eqref{twisted with x} interchanges the roles of $u_0$ and $u_1$ with $v_0$ and $v_1$, respectively, and puts $x^{-1}$ in place of $x$. If $\hat{\pi}_i$ is the $i$--th piece in the rank $i$ piece decomposition of $x^{-1}$, then $\hat{\pi}_i = \pi_{k-i}^{-1}$, and this means that to satisfy \eqref{shift} with $\hat{\pi}_i$ instead, we use $-e$ instead of $e$.
	So in Cases~\ref{case2}b and \ref{case2}c below, when $e$ is negative, swapping the roles of $u_i$ and $v_i$ accordingly will give the structure of $X^{-1}$, and that of $X$ will then be apparent.

\case{\ref{case2}a} When $e =0$. 

If we remove  the pieces $\pi_{a+1}, \ldots , \pi_{b-1}$  from $x$ leaving $X_0 :=  L \pi_a' \pi_b' R$, then  
$\varphi^{-p}(u_0  X_0  v_0^{-1})   =  u_1^{-1}  X_0 v_1$  in $F$ by   Lemma~\ref{easy insert}. Next we will    replace $\pi_a'$ with $\hat{\pi}_a'$, and $\pi_b'$ with $\hat{\pi}_b$, which come from Corollary~\ref{cor:fbc shared prefix of words}, as explained below, giving $X := L \hat\pi_a' \hat\pi_b' R$.  We will show this too satisfies \eqref{twisted with X}.

As  $\pi_a'$ is a common suffix of $\pi_a$ and $\varphi^{p}(\pi_a)$  (and so a common prefix of $\pi_a^{-1}$ and $\varphi^{-p}(\pi_a)^{-1}$),  Corollary~\ref{cor:fbc shared prefix of words} gives us a word $\hat{\pi}_a$ so that
\begin{enumerate}\renewcommand{\theenumi}{$\mathcal{A}$\arabic{enumi}}
	\item \label{common suffix break} if $\alpha\beta$ is the free reduction of $\pi_a\varphi^{-p}(\pi_a)^{-1}$, with $\pi_a = \alpha \pi_a'$ and $\varphi^{-p}(\pi_a) = \beta^{-1} \pi_a'$ as words, 
	then $\alpha \beta$ is also the free reduction of  $\hat{\pi}_a\varphi^{-p}(\hat{\pi}_a)^{-1}$ in $F$, with $\hat\pi_a = \alpha \hat\pi_a'$ and $\varphi^{-p}(\hat\pi_a) = \beta^{-1} \hat\pi_a'$ as words, and
	\item \label{common suffix form} the common suffix $\hat{\pi}_a'$ of $\hat{\pi}_a$ and $\varphi^{-p}(\hat{\pi}_a)$ is of the form $\hat{\pi}_a' = S_i\cdots S_3 S_1$, where
	\begin{itemize}
		\item $S_1^{-1}$ is a prefix of $\varphi^k(a_t)$ for some $t\leq i$ and $\abs{k} \leq A_i\left(\abs{ \pi_a\varphi^{-p}(\pi_a)^{-1}}_H + p\right)$,
		\item $S_j$ is a subword of $\varphi^{-p}(a_j)$ for $j=3,\ldots, i$. 
	\end{itemize} 
\end{enumerate} 

  	First we check that $S=S_i\cdots S_3 S_1$ fits the scheme of the proposition.  To do this we need to bound $\abs{\pi_a\varphi^{-p}(\pi_a)^{-1}}_H$ so that \eqref{common suffix form} leads to the required bound on $\abs{k}$. We could use \eqref{structure of L} and \eqref{structure of Y} alongside Proposition~\ref{prop:fbc subword length}, but we can do better  as follows.
 By \eqref{what cancellation}, $LY^{-1} = u_1\varphi^{-p}(u_0)$ in $F$.	The last letters of $L$ and $Y$ must be different, since otherwise that letter could be added into $\Pi$.	So $LY^{-1}$ is reduced.
	The free reduction of $\pi_a (\pi_a')^{-1}$ is a suffix   of $L$, and that of $\varphi^{-p}(\pi_a)(\pi_a')^{-1}$ is a suffix  of $Y$.  
	Hence the free reduction of $\pi_a  \varphi^{-p}(\pi_a)^{-1}$ is a subword of the free reduction of $u_1\varphi^{-p}(u_0)$.
	So Proposition~\ref{prop:fbc subword length}  gives the first  inequality of:
	$$	\abs{\pi_a  \varphi^{-p}(\pi_a)^{-1}}_H  \ \leq \  (2m+1)\abs{u_1  \varphi^{-p}(u_0)}_H   \ \leq \  (2m+1)\left( \abs{u_0}_F +\abs{u_1}_F  +2p\right). $$
	We   deduce that there is a constant $C>0$ such that
	$$\abs{k}  \ \leq \  C \left( \abs{u_0}_F + \abs{u_1}_F + p \right).$$ 
	
  Similarly, $\pi'_b$ is a common prefix of $\pi_b$ and $\varphi^{-p}(\pi_{b})$, so Corollary~\ref{cor:fbc shared prefix of words}, tells us that there is  word $\hat\pi_b$ such that
 \begin{enumerate}\renewcommand{\theenumi}{$\mathcal{B}$\arabic{enumi}}
 	\item \label{common suffix break b}   
    if $\gamma\delta$ is the free reduction of $\pi_b^{-1}\varphi^{-p}(\pi_b)$, with $\pi_b = \pi_b' \gamma^{-1}$ and $\varphi^{-p}(\pi_b) = \pi_b' \delta$ as words, 
    then $\gamma \delta$ is also the free reduction of  $\hat{\pi}_b^{-1}\varphi^{-p}(\hat{\pi}_b)$ in $F$, with $\hat\pi_b = \hat\pi_b' \gamma^{-1}$ and $\varphi^{-p}(\hat\pi_b) = \hat\pi_b' \delta$ as words,  and
   \item \label{common suffix form a} the common prefix $\hat\pi_b'$ of $\hat\pi_b$ and $\varphi^{-p}(\hat\pi_b)$ is of the form $\pi_b' = P_1P_3\cdots P_i$, where 
 \begin{itemize}
 	\item $P_1$ is a prefix of $\varphi^{k'}(a_t)$ for some $t\leq i$ and $\abs{k'} \leq A_i\left(\abs{\pi_b^{-1}\varphi^{-p}(\pi_b)}_H + p\right)$,
 	\item $P_j$ is a subword of $\varphi^{-p}(a_j^{-1})$ for $j=3,\ldots, i$.
 \end{itemize} 
\end{enumerate}
Similar reasoning to the above give us $\abs{k'} \leq C ( \abs{v_0}_F +  \abs{v_1}_F + p)$.
It follows that $X  = L  S P R$ has form \eqref{form of x 2}, with $M$ the empty word.

To conclude, we need to prove $X$ will satisfy \eqref{twisted with X}.
First we work on the left. Using the notation from \eqref{common suffix break}, we have $L = \pi_1 \cdots \pi_{a-1} \alpha$ and $Y = \varphi^{-1}(\pi_1 \cdots \pi_{a-1}) \beta^{-1}$,
so $Y = \varphi^{-p}(L\alpha^{-1})\beta^{-1}$.
Then by \eqref{what cancellation}, 
$\varphi^{-p}(u_0 L \alpha^{-1}) \beta^{-1} 
= u_1^{-1} L$.
Rearranging  and using that $\alpha\hat\pi_a' = \hat\pi_a$ and $\beta^{-1}\hat\pi_a = \varphi^{-p}(\pi_a)$, we get
$$\varphi^{-p}(u_0 L \hat\pi_a') 
\ = \  u_1^{-1} L \beta \varphi^{-p}(\alpha \hat\pi_a')
\ = \  u_1^{-1} L \beta \varphi^{-p}(\hat\pi_a)
\ =  \ u_1^{-1} L \beta \beta^{-1} \hat\pi_a'
\ =  \ u_1^{-1} L \hat\pi_a'.
$$
 
Similar calculations on the right yield $\varphi^{-p}(\hat\pi_b' R u_0^{-1}) = \hat\pi_b' R v_1$. The left and right, working together, give us \eqref{twisted with X}.

\case{\ref{case2}b} When $\abs{e}>b-a-1$.

	First assume $e>0$.
	Write $\Pi = SMP,$
	where \blue{$S = \pi_a'$,}
	$M =M_2= \pi_{a+1} \cdots \pi_{b-1}$,
	and $P = \pi_b'$.
	 Then $M_2 = \varphi^{-p}(\pi_{a+1+e} \cdots \pi_{b-1+e})$ by \eqref{shift}, and $\varphi^p(M_2) = \pi_{a+1+e} \cdots \pi_{b-1+e}$ is a subword of $\varphi^{-p}(v_0)^{-1}$ by \eqref{structure of R}.
	Meanwhile,  $S$  is a subword of $\pi_a$, and $\varphi^{-p}(\pi_a)$ is a subword of $(u_1\varphi^{-p}(u_0))^{-1}$ by \eqref{shift outside2}.
	Finally, $P$ is a subword of $\varphi^{-p}(\pi_{b+e})$, 
	and $\hat{P} = \pi_{b+e}$ is a subword of $\varphi^{-p}(v_0^{-1})$ by \eqref{shift outside4}.

	Now assume that $e<0$. Then the above gives the structure for $\Pi^{-1}$, after swapping the roles of $u_i$ and $v_i$.
	Hence we get $\Pi = SMP$, where
	$S$ is a subword of  $\varphi^{-p}(\hat{S})$ where $\hat{S}$  is a subword of $\varphi^{-p}(u_0)$,
	$M=M_2$ and  $\varphi^{-p}(M_2)$ is a subword of $v_1\varphi^{-p}(v_0)$, and
	$P$ is a subword of $\hat{P}$, and $\varphi^{-p}(\hat{P})$ is a subword of $v_1\varphi^{-p}(v_0)$.

\case{\ref{case2}c} When  $0 < \abs{e} \leq b-a-1$.

Suppose $e>0$.  
We have  $$\Pi  \ = \  \pi_a' \pi_{a+1} \cdots \pi_{b-1}\pi_b'.$$  The number of pieces in $\Pi$ aside from $\pi_a'$  and $\pi_b'$ is $b-a-1$.    
Let $q$ be the maximal integer such that $qe \leq b-a-1$.    
Let $$\pi \  = \   \pi_{a+1} \cdots \pi_{a+e}.$$  
By repeated applications of \eqref{shift},
	$$\Pi = S M P \ \  \text{ in } F$$
	where $M= M_1 M_2$ and 
	\begin{align*}
	S  & = \  \pi_a' \\ 
	M_1 & =  \ \pi \varphi^p(\pi) \cdots \varphi^{p(q-1)}(\pi)     \\ 
	M_2 & =  \  \varphi^{pq}(\pi_{a+1} \cdots \pi_{b-1-qe}) \\ 
	P  & = \  \pi_b'. 
 	\end{align*}

 Then $S$ and $P$ are as in Case~\ref{case2}b. 
   Notice that $\varphi^{-p}(\pi)$ is a concatenation of a subword of $Y$  (so, by \eqref{structure of Y}, a  subword of the freely reduced form of  $(u_1\varphi^{-p}(u_0))^{-1}$) with $S= \pi_a'$.
	By \eqref{shift}, $\varphi^{p}(M_2) = \pi_{a+1+(q+1)e} \cdots \pi_{b-1+e}$.
    Note that $a+1+(q+1)e \ne b$, by our choice of $q$.
	So $a+1+(q+e)>b$, which implies $\varphi^p(M_2)$ is a subword of $\varphi^{-p}(v_0^{-1})$ by \eqref{structure of R}.  

We now begin the work necessary to establish the required  bound on $q$. 

First, for $a \leq c \leq c' < a+e$, we have that $\varphi^{-p}(\pi_c \cdots \pi_{c'})$ is a subword of $Y$, and so of the freely reduced form of  $(u_1\varphi^{-p}(u_0))^{-1}$ by \eqref{structure of Y}.
Also  $\varphi^{-p}(\pi_{a+e})$ is a product of subwords of $(u_1\varphi^{-p}(u_0))^{-1}$ and $\pi_a$.
So, by Proposition~\ref{prop:fbc subword length}, there is a constant $C>0$, depending only on $m$, such that 
\begin{equation}\label{eq:piece bound u}
\abs{\pi_{c} \cdots \pi_{c'}}_H \ \leq \  C \left(\abs{u_1 \varphi^{-p}(u_0)}_H +  p\right) \ \ \textrm{for $a \leq c \leq c' \leq a+e$.}
\end{equation}
Similarly,  
\begin{equation}\label{eq:piece bound v}
\abs{\pi_d \cdots \pi_{d'}}_H \ \leq \  C\left( \abs{v_1\varphi^{-p}(v_0)}_H  + p \right) \ \ \textrm{ for $b \leq d \leq d' \leq b+e$.}
\end{equation}

We now bound $qp$.
Let $a<c\leq a+e$
and choose $b\leq d < b+e$ so that $\pi_d = \varphi^{q'p}(\pi_{c})$, where $q'$ is $q$ or $q+1$ according to whether $c >b-1-qe$ or not.
Suppose $\varphi(\pi_{c}) \neq \pi_{c}$  and the rank of $\pi_c$ is $j\leq i$.  We will establish an upper bound on $q'p$
	by combining an upper bound on $\abs{\pi_d}_H$ from \eqref{eq:piece bound v}  with an understanding of how fast $\pi_c$ can grow under iterates of $\varphi^{-1}$ from Proposition~\ref{prop:lower bound distortion}.

 	We claim that we can choose such $c$ so that $j=i$ (that is,  so that $\pi_c$ has rank $i$).
	First assume that $e = 1$. Then \eqref{shift} gives $\pi_c = \varphi^{-p}(\pi_{c+1})$. In particular, both $\pi_c$ and $\pi_{c+1}$ have the same rank, which must therefore be $i$, since in a pair of adjacent pieces of a rank--$i$ decomposition of a word, at least one must have rank $i$.
	So we may assume $e>1$.
	Since we assume $\varphi(\pi_c) \neq \pi_c$, its rank satisfies $j\ge 2$ by Lemma~\ref{lem:fixed pieces}.
	Suppose $j<i$.
	Then both the the neighbours $\pi_{c+1}$ and $\pi_{c-1}$ of $\pi_c$ must have rank $i$.
	But $i > j \ge 2$, so  $\varphi(\pi_{c+1}) \neq \pi_{c+1}$ and $\varphi(\pi_{c-1}) \neq \pi_{c-1}$ by Lemma~\ref{lem:fixed pieces} again.
	The inequality $a < c\pm 1 \leq a+e$ holds for at least one of $c+1$ or $c-1$. We replace $c$ with the corresponding number and may therefore assume $j=i$.

If $(q'p)^{i-1} < \frac{1}{C_i}{\abs{\pi_{c}}_F}$, then  $C_i^\frac{1}{i-1}q'p
\ < \  \abs{\pi_{c}}_F^\frac{1}{i-1}$, and so
 Corollary~\ref{cor:rank i distortion} and then \eqref{eq:piece bound u} imply
$$C_i^\frac{1}{i-1}q'p
\ < \ {K_i}^\frac{1}{i-1}  \abs{\pi_{c}}_H
\ \leq \ {K_i}^\frac{1}{i-1}   C\left( \abs{ u_1 \varphi^{-p}(u_0)}_H + p \right).$$
If, on the other hand, $(q'p)^{i-1} \geq \frac{1}{C_i}{\abs{\pi_{c}}_F}$, then, we may apply Proposition~\ref{prop:lower bound distortion}   to $\pi_d = \varphi^{q'p}(\pi_{c})$, giving
$$\abs{ \pi_d}_F  \ = \   \abs{\varphi^{q'p}(\pi_{c})}_F \ \geq  \ \left(C_i^{\frac{1}{i-1}}q'p - {\abs{\pi_{c}}_F}^{\frac{1}{i-1}}\right)^{i-1}.$$
Rearranging and then combining this with Corollary~\ref{cor:rank i distortion} and inequalities \eqref{eq:piece bound u} and \eqref{eq:piece bound v} yields 
$$C_i^{\frac{1}{i-1}}q'p 
\ \leq \ {K_i}^{\frac{1}{i-1}} \left( \abs{\pi_{c}}_H 
+ \abs{\pi_{d}}_H\right)
\ \leq \  {K_i}^{\frac{1}{i-1}} C \left( \abs{ u_1 \varphi^{-p}(u_0)}_H + \abs{  v_1 \varphi^{-p}(v_0)}_H + 2p\right).$$  Increasing the value of $C$ if necessary, we then get
\begin{equation}
\label{eq:bound on qp}
qp \ \leq \   C (\abs{u_0}_H + \abs{u_1}_H + \abs{v_0}_H + \abs{v_1}_H + p).
\end{equation} 
 
So, provided there is some such $c$ with $\varphi(\pi_c) \neq \pi_c$, we have the required bound on $qp$, and taking $\hat{\Pi} = \Pi = SMP$ gives an $X$ satisfying \eqref{form of x 2}.

If, on the other hand, $\pi_{c}$ is fixed by $\varphi$ for all $a<c\leq a+e$
then we may cut a big chunk out of $\Pi$, since then $\varphi^p(\pi)=\pi$ and
$M = \pi^{pq}\varphi^{pq}(\pi_a \cdots \pi_{b-1-qe})$.
By Lemma~\ref{easy insert} we can remove $\pi^{pq}$,  and then  $\hat{\Pi} =  S M P$, where $M=M_2$, gives an $X$ satisfying the proposition.

When $e<0$, we get the structure of $\Pi^{-1}$ and of $\hat{\Pi}^{-1}$ from the above argument, once  the roles of $u_i$ and $v_i$ have been swapped.
As above, $S$ and $P$ will be as in Case~\ref{case2}b.
For $M$, we need to change the order of $M_1$ and $M_2$, but we have to be careful as taking the inverse of $M_1$ changes its structure.
The easiest way to express this is to say $M = M_2^{-1}M_1^{-1}$, with $M_1$ and $M_2$ obtained as above, but with the $u_i$  and $v_i$ exchanged.

\case{\ref{case1}} $\Pi=\pi_a'$.

By \eqref{shift setup},  $\Pi$ is a subword of $\varphi^{-p}(\pi_{a+e})$.   
If $e<0$,  then    by \eqref{shift outside1}, $\pi_{a+e}$ is a subword of $\varphi^{-p}(u_0)$. 
So $\hat{\Pi} = \Pi$   satisfies the conditions of form \eqref{form of x 2} of the proposition with  
$S = \Pi$,  $\hat{S} = \pi_{a+e}$,  and $M$ and $P$ both the empty word.  If $e  >  0$, then we take  $\hat{\Pi} = \Pi$  and it similarly satisfies the conditions  with $S$  and $M$  both the empty word and $P = \Pi$, which is a subword of $\varphi^{-p}(\pi_{a+e})$ and $\hat{P} =  \pi_{a+e} $  is a subword of  $\varphi^{-p}(v_0^{-1})$.

On the other hand, assume $e = 0$.
If $\Pi$ is either a prefix or a suffix of $\pi_a$ then we can apply Corollary~\ref{cor:fbc shared prefix of words}.
If $\Pi$ is a prefix, we replace it with $\hat{\Pi} = P = P_1 P_3 \cdots P_i$, as in Corollary~\ref{cor:fbc shared prefix of words}.
The proof that \eqref{twisted with X} holds, and of the bound on the $k$  in $P_1$ are the same as for  Case \ref{case2}a (treating $\pi_a'$ as the empty word  and $b=a+1$).
If $\Pi$ is instead a suffix, replace it with $\hat{\Pi} = S = S_i \cdots S_3S_1$ using Corollary~\ref{cor:fbc shared prefix of words}, and a similar check gives \eqref{twisted with X} and a corresponding bound on $k'$.

What remains is to consider when $\Pi$ is not a prefix or suffix of $\pi_a$. Then it has rank $j<i$. This is where the form \eqref{form of x 3} occurs.
As reduced words, write $\pi_a=\hat{u}_0\hat{x}\hat{v}_0^{-1}$ and $\varphi^{-p}(\pi_a) = \hat{u}_1^{-1}\hat{x}\hat{v}_1$, where $\hat{x} = \pi_a' = \Pi$.    
So $\rank(\hat{x}) < i$ and $\varphi^{-p}(\hat{u}_0\hat{x}\hat{v}_0^{-1}) = \hat{u}_1^{-1}\hat{x}\hat{v}_1$, as required. 
Since $\hat{u}_0$ is a suffix of $L$ and $\hat{u}_1$ is a prefix of $Y^{-1}$, and there is no cancellation between $L$ and $Y^{-1}$ by the definition of $\Pi$, we can deduce from \eqref{what cancellation} that $\hat{u}_0\hat{u}_1$ is reduced and is a subword of $\varphi^{-p}(u_0)$. 
\end{proof}  

\begin{proof}[Proof of Lemma~\ref{lem:annoying iteration}]
		By hypothesis
	\begin{equation}\label{eq:Xhat}
	s^p \ \hat{u}_0 \ \hat{X} \ \hat{v}_0^{-1} \ s^{-p} \ =\  \hat{u}_1^{-1} \ \hat{X} \ \hat{v}_1  \ \ \text{ in } H.
	\end{equation}
	By \eqref{what cancellation},
	$s^p \ u_0   \ s^{-p} \  Y \  \hat{u}_1 = u_1^{-1} L\hat{u}_1$ in $H$, which together with  \eqref{eq:Xhat} gives	 
		\begin{equation}\label{eq:Xhatmult}
	 ( s^p \ u_0 \ s^{-p} \  Y \  \hat{u}_1 ) \ (   s^p \ \hat{u}_0 \ \hat{X} \ \hat{v}_0^{-1} \ s^{-p}) 
	\ = \ (u_1^{-1} \ L \  \  \hat{u}_1)  \ ( \hat{u}_1^{-1} \ \hat{X} \ \hat{v}_1)  \ \ \text{ in } H.
		\end{equation}
	By hypothesis, $X$ has form \eqref{form of x 3}, and as per Case~\ref{case1} of our proof of  Proposition~\ref{conjugator-like problem in F}, 	 $L = \pi_1 \cdots \pi_{a-1} \hat{u}_0$ and $Y = \varphi^{-p} ( \pi_1 \cdots \pi_{a-1}) \hat{u}_1^{-1}$. 
	Comparing these expressions for $L$ and $Y$ we get that  in $H$ we have  $L = s^{-p} Y \hat{u}_1 s^p \hat{u}_0$, the right-hand side of which is a substring of the left side of \eqref{eq:Xhatmult}.  
	So  substituting accordingly into the left and cancelling the $\hat{u}_1  \hat{u}_1^{-1}$ from the right, we get 	
	\begin{equation}\label{eq:leftfixedup}
s^p \ u_0 \ L \ \hat{X} \ \hat{v}_0^{-1} \ s^{-p}  
	 =   u_1^{-1} \ L \    \ \hat{X} \ \hat{v}_1   \ \ \text{ in } H.
	 \end{equation} 

	In the same manner as we derived  \eqref{eq:leftfixedup} from \eqref{eq:Xhat}  using \eqref{what cancellation},
	 calculations on the right-hand ends using the equation $Z\varphi^{-p}(v_0^{-1}) = Rv_1$  will  derive      \eqref{LR in w} from \eqref{eq:leftfixedup}.
	
	The final part of the lemma follows from the discussion in Section~\ref{reductions} or by direct calculation.
\end{proof}

\section{Solving the $\H$-twisted conjugacy problem} \label{H-twisted section2}

\begin{lemma}\label{lem:chunk lenght bound}
There exists a constant $K>0$, depending only on $m$, such that  the lengths in $H$ of
the elements $U_1$, $U_2$, $U_3$, $V$,   $L$, $S$, $P$, $M_2$,  $R$, $S$, $\hat{u}_0$, $\hat{u}_1$, $\hat{v}_0$,   $\hat{v}_1$, and $\pi$  
that arise in Proposition~\ref{conjugator-like problem in F} are all at most  $K \Sigma$, where   
$$\Sigma \ =  \ (\abs{u_0}_H  + \abs{u_1}_H  + \abs{v_0}_H  + \abs{v_1}_H  + p).$$
Moreover, $\abs{M_1}_H \leq K \Sigma^2.$
\end{lemma}

\begin{proof} The $K \Sigma$ upper  bounds all follow from applying Proposition~\ref{prop:fbc subword length} to the descriptions of the words and the associated bounds given in Proposition~\ref{conjugator-like problem in F}, noting that $\abs{\varphi^{j}(w)}_H \leq 2|j| + \abs{w}_H$ and that the $i$ in the proposition is at most $m$.

 For the bound on the length of $M_1$, observe that 
	\begin{align*}
	M_1 
	& \ = \  \pi \varphi^p(\pi) \cdots \varphi^{p(q-1)}(\pi) \\
	& \ = \ \pi (s^{-p} \pi s^p)  (s^{-2p} \pi s^{2p}) \cdots (s^{-p(q-1)} \pi s^{p(q-1)}) \\
	& \ = \   (\pi s^{-p})^{q} s^{qp}. 
	\end{align*}  
	Combining \eqref{eq:bound on qp} with $\abs{\pi}_H \leq K \Sigma$ and adjusting $K$ suitably gives $\abs{M_1}_H \leq K \Sigma^2.$  
\end{proof}

Recall that the  $\H$-twisted conjugacy problem asks: given reduced words $\tilde{u}$, $\tilde{v}$ on $a_1^{\pm 1}, \ldots , a_m^{\pm 1}$  and an integer $p >0$, do there exist $0 \leq r < p$ and   words $x , u_0, v_0, u_1, v_1 \in F$ such that $\tilde{u} = u_0u_1$ and $\varphi^{-r} (\tilde{v}) = v_0v_1$, as words, and 
		\begin{equation*}   
	    \varphi^{-p}(u_0  x  v_0^{-1})  \ = \  u_1^{-1} x v_1  \  \text{  in }  \   F? 
	  	\end{equation*}

\begin{lemma} \label{inductive alg}
For all $i=1, \ldots, m$, there exists an algorithm  that, with input any $(p, \tilde{u}, \tilde{v})$ for which the $\H$-conjugacy problem   has a solution $(r, x , u_0,v_0,u_1,v_1)$  with $\rank(x) =i$,  will exhibit \emph{some} solution $(r, X, u_0,v_0,u_1,v_1)$;
the running time of this algorithm is bounded above
by a polynomial in $p + \abs{\tilde{u}}_H + \abs{\tilde{v}}_H$ (where
the implied constants depend only on the rank $m$ of $F$).    
\end{lemma}

\begin{proof}
 In the following, when we refer to \emph{polynomial bounds}, we 
 will always mean upper bounds that are
 polynomial in $p + \abs{\tilde{u}}_H + \abs{\tilde{v}}_H$.   We induct on $i$. 

Proposition~\ref{conjugator-like problem in F} tells us that in the case $i=1$, there is solution $(r, X , u_0,v_0,u_1,v_1)$  in which $X$ takes the form \eqref{form of x 0}.  We can find one such solution in polynomial time as follows.  We list all the   $(r, u_0,  u_1, v_0,  v_1)$   such that $0 \leq r < p$ and  $\tilde{u} = u_0u_1$ and $\varphi^{-r} (\tilde{v}) = v_0v_1$ as words---there are polynomially many and the words involved all have polynomially bounded length.  For each we list all $U_1$, $U_2$, $U_3$ and $V$ as per \eqref{form of x 0}---again, polynomially many possibilities---and we  check whether of not $\varphi^{-p}(u_0  X  v_0^{-1})    =    u_1^{-1} X v_1$ in $F$ for $X= U_1U_2U_3V$.   

Now suppose that there exists a solution in which $\rank(x)  = i >1$ and that the lemma holds when there exists a solutions in which $x$ has lower rank.     

Again list  the  polynomially many $(r, u_0,  u_1, v_0,  v_1)$  such that $0 \leq r < p$ and  $\tilde{u} = u_0u_1$ and $\varphi^{-r} (\tilde{v}) = v_0v_1$ as words.  For each,   list the polynomially many  words $X$ of the form  \eqref{form of x 1}   or \eqref{form of x 2} of Proposition~\ref{conjugator-like problem in F} and  check whether  $\varphi^{-p}(u_0  X  v_0^{-1})    =    u_1^{-1} X v_1$ in  $F$.  If this fails to turn up a solution $(r, X , u_0,v_0,u_1,v_1)$, then the proposition tells us that  there must be one in which $X$ has the form \eqref{form of x 3}.  Accordingly,  for each of the $(r, u_0,  u_1, v_0,  v_1)$, list  all the polynomially many $(L,R,\hat{u}_0, \hat{u}_1, \hat{v}_0, \hat{v}_1)$  satisfying the conditions of Proposition~\ref{conjugator-like problem in F}.   (These $L, R, \hat{u}_0, \hat{u}_1, \hat{v}_0, \hat{v}_1$  have polynomially bounded length.)  

We are considering polynomially many  $(r, u_0,  u_1, v_0,  v_1)$, and for each one there are polynomially many  $(L,R,\hat{u}_0, \hat{u}_1, \hat{v}_0, \hat{v}_1)$, so this amounts to polynomially many possibilities in total.  For one of them, there is an $\hat{x}$ with $\rank(\hat{x}) <i$ such that $$\varphi^{-p}(\hat{u}_0  \hat{x}  \hat{v}_0^{-1})   =  \hat{u}_1^{-1}  \hat{x} \hat{v}_1 \  \text{  in }  \   F$$  and 
\begin{equation*}    
	    \varphi^{-p}(u_0  L \hat{x} R  v_0^{-1})  \ = \  u_1^{-1} L \hat{x} R v_1  \  \text{  in }  \   F. 
	  	\end{equation*}  
By induction we have  a polynomial time algorithm which we can run (in polynomial time overall) on every one of these possibilities, and for one of them it will exhibit some $\hat{X}$ such that $$\varphi^{-p}(\hat{u}_0  \hat{X}  \hat{v}_0^{-1})   =  \hat{u}_1^{-1}  \hat{X} \hat{v}_1 \  \text{  in }  \   F.$$  
So Lemma~\ref{lem:annoying iteration} gives us that $\varphi^{-p}(u_0  L \hat{X} R  v_0^{-1})    =    u_1^{-1} L \hat{X} R v_1$   in $ F$,  and thereby we get a solution $(r, X , u_0, v_0, u_1, v_1)$ where $X = L \hat{X} R$.  
\end{proof}

\begin{cor}[$\H$-twisted conjugacy complexity] \label{conjugator-like problem in F solution}
There is an algorithm that takes as input an integer $p >0$   and  reduced words $\tilde{u}$ and $\tilde{v}$ on $a_1^{\pm 1} , \ldots , a_m^{\pm 1}$ and determines whether or not there exists a solution $(r, x , u_0,v_0,u_1,v_1)$ to the $\H$-twisted conjugacy problem.  If a solution exists, it exhibits one. 
The running time of the algorithm is bounded above be a
 polynomial function of  $p + \abs{\tilde{u}}_H + \abs{\tilde{v}}_H$.  
\end{cor}

\begin{proof}  
Run the algorithms of Lemma~\ref{inductive alg} for $i=1, \ldots, m$ on input $(p, \tilde{u}, \tilde{v})$.  The time that it takes each to  halt is  bounded above 
by a polynomial in $p + \abs{\tilde{u}}_H + \abs{\tilde{v}}_H$.    If there exists a solution, one of them  will exhibit it. 
\end{proof}

Lemma~\ref{lem:chunk lenght bound} gives us many of the ingredients for  the desired linear upper bound on the conjugator length of $H$, but we will need a way around the quadratic bound on $\abs{M_1}_H$.  Accordingly, we will manipulate the form of the conjugator in the case \eqref{form of x 2} of Proposition~\ref{conjugator-like problem in F}, which is where  $M_1$ appears.

\begin{lemma}\label{lem:swap for a linear conjugator}
	Suppose $u = \tilde{u} s^p$ and $v = \tilde{v}s^p$ are conjugate elements of $H$,
	and there is a solution  $(r, x , u_0,v_0,u_1,v_1)$ to the $\H$-twisted conjugacy problem for  $(p, \tilde{u}, \tilde{v})$  in which $x$ has form \eqref{form of x 2}. Let $q$ be as in Proposition~\ref{conjugator-like problem in F} (in the form of $M_1$).
	Then either $u_0 LS s^{pq} M_2 P R v_0^{-1} s^r$   or $u_0 LSM_2^{-1} s^{-pq} PRv_0^{-1} s^r$    conjugates  $u$ to $v$. 
\end{lemma}

\begin{proof}
	We are in the setting of Case \ref{case2}c of our proof of Proposition~\ref{conjugator-like problem in F}.
	Assume that $e>0$.  Let $w = u_0 LS s^{pq} M_2 P R v_0^{-1} s^r$. We will show that $uw=wv$ in $H$.

	We have that  $u= u_0 u_1 s^p$ and $v = \varphi^r(v_0v_1)s^p$ in $H$.  
	Also $M_2 = \varphi^{pq}(\pi_{a+1} \cdots \pi_{b-1-qe})$, which equals $\pi_{a+1+qe} \cdots \pi_{b-1}$ 
	by \eqref{shift}.
	Plugging this and the other ingredients into $w$, we get
	$$w  \ = \  u_0 \ \pi_1 \cdots \pi_a \ s^{pq} \ \pi_{a+1+qe} \cdots \pi_k \ v_0^{-1} \ s^r.$$
	Then, by repeatedly applying \eqref{shift} and the identity $s^{-1} g s = \varphi(g)$ for $g\in F$, we get
	\begin{align*}
	uw 
	& \ = \  u_0 \ u_1 \ s^p \ u_0 \ \pi_1 \cdots \pi_a \ s^{pq} \ \pi_{a+1+qe} \cdots \pi_k \ v_0^{-1}\ s^r \\
	& \ = \  u_0 \ u_1 \ \varphi^{-p}(u_0 \ \pi_1 \cdots \pi_a) \ s^{p(q+1)} \ \pi_{a+1+qe} \cdots \pi_k \ v_0^{-1}\ s^r \\
	& \ = \  u_0 \ u_1 \ \varphi^{-p}(u_0 \ \pi_1 \cdots \pi_a) \  \varphi^{-p(q+1)}(\pi_{a+1+qe} \cdots \pi_{b-1})  \  s^{p(q+1)} \ \pi_{b} \cdots \pi_k \ v_0^{-1}\ s^r \\
	& \ = \  u_0 \ u_1 \ \varphi^{-p}(u_0 \ \pi_1 \cdots \pi_a) \  \varphi^{-p}(\pi_{a+1} \cdots \pi_{b-1-qe})  \  s^{p(q+1)} \ \pi_{b} \cdots \pi_k \ v_0^{-1}\ s^r \\
	& \ = \  u_0 \ u_1 \ \varphi^{-p}(u_0 \ \pi_1 \cdots  \pi_{b-1-qe})  \  s^{p(q+1)} \ \pi_{b} \cdots \pi_k \ v_0^{-1}\ s^r \\
	& \ =  \  u_0 \ u_1 \ \varphi^{-p}(u_0 \ \pi_1 \cdots  \pi_{b-1-qe})  \  \varphi^{-p(q+1)}(\pi_b \cdots \pi_{a+(q+1)e}) \  s^{p(q+1)} \ \pi_{a+(q+1)e+1} \cdots \pi_k \ v_0^{-1}\ s^r \\
	& \ = \  u_0 \ u_1 \ \varphi^{-p}(u_0 \ \pi_1 \cdots  \pi_{b-1-qe} \ \pi_{b-qe} \cdots \pi_{a+e}) \  s^{p(q+1)} \ \pi_{a+(q+1)e+1} \cdots \pi_k \ v_0^{-1}\ s^r.
	\end{align*}
	By \eqref{what cancellation}, $\varphi^{-p}(u_0 \ \pi_1 \cdots \pi_{a+e}) = u_1^{-1} \pi_1 \cdots \pi_a$.
	Hence
	\begin{equation} \label{eq:uw}
	uw
	= u_0 \ \pi_1 \cdots \pi_a \  s^{p(q+1)} \ \pi_{a+(q+1)e+1} \cdots \pi_k \ v_0^{-1}\ s^r.
	\end{equation}
	Similar calculations give:
	\begin{align*}
	wv
	& \ = \  u_0 \ \pi_1 \cdots \pi_a \ s^{pq} \ \pi_{a+1+qe} \cdots \pi_k \ v_0^{-1} \ s^r \ \varphi^r(v_0v_1) \ s^p\\
	& \ = \  u_0 \ \pi_1 \cdots \pi_a \ s^{pq} \ \pi_{a+1+qe} \cdots \pi_k \ v_0^{-1} \ v_0 \ v_1 \ s^{p+r}\\
	& \ = \  u_0 \ \pi_1 \cdots \pi_a \ s^{p(q+1)} \ \varphi^{p}(\pi_{a+1+qe} \cdots \pi_{b-1} \ \pi_{b} \cdots  \pi_k \  v_1) \ s^{r}.
	\end{align*}
	The corresponding fact to \eqref{what cancellation} concerning $R$ and $Z$ is that $Z\varphi^{-p}(v_0^{-1})$ freely reduces to $R v_1$.
	It implies $\varphi^{-p} ( \pi_{b+e} \cdots \pi_k \ v_0^{-1} ) = \pi_b \cdots \pi_k v_1$.
	Together with one final application of \eqref{shift} to $\varphi^p(\pi_{a+1+qe} \cdots \pi_{b-1})$, this gives
	$$
	wv  \ = \  u_0 \ \pi_1 \cdots \pi_a \ s^{p(q+1)} \ \pi_{a+1+(q+1)e} \cdots \pi_{b+e-1} \  \pi_{b+e} \cdots \pi_k \ v_0^{-1}  \ s^{r},
	$$
	which equals $uw$ by equation~\eqref{eq:uw}.	
	
The proof when $e<0$ is similar, giving $uw=wv$ in $H$ for $w = u_0 LSM_2^{-1} s^{-pq} PRv_0^{-1} s^r$.
\end{proof}

\section{Completing our proof of Theorem~\ref{CL of free-by-cyclic}} \label{sec:CL of H}

	We will establish a linear  upper bound on the conjugator length of $H$.
	Suppose $u$, $v$ and $w$  are words on  $a_1^{\pm 1}, \ldots, a_m^{\pm 1}, s^{\pm 1}$  such that $uw=wv$.   We will show that there is a word $W$ on  $a_1^{\pm 1}, \ldots, a_m^{\pm 1}, s^{\pm 1}$ such that $uW=Wv$ and $\ell(W)$ at most a constant times  $\abs{u}_H + \abs{v}_H$.
	
	Write the normal forms  of $u$, $v$ and $w$ as  $\tilde{u}s^p$, $\tilde{v} s^p$ and $\tilde{w}s^r$, respectively.
		
	Following the discussion of Section~\ref{reductions}, if $p=0$ then we are in the 0-twisted conjugacy case, and we find   $W$ via Proposition~\ref{0 twisted conj problem} \eqref{case:0-twisted CL}.
	
	When $p\ne 0$, as we can replace $u$ and $v$  by their inverses if necessary, we may assume $p>0$.
	As discussed in Section~\ref{reductions} we may also replace $w$ with $u^jw$ so we can assume $0\leq r < p$.
	By Proposition~\ref{prop:fbc subword length},  $\abs{u'}_H \leq (2m+1)\abs{u}_H$  and  $\abs{v'}_H \leq (2m+1)\abs{\varphi^{-r}(\tilde{v})}_H \leq (2m+1)(2r+\abs{v}_H)$ for any subwords $u'$ of $\tilde{u}$ and $v'$ of $\varphi^{-r}(\tilde{v})$.      
	We also have $0\leq r<p \leq  \abs{u}_H$. 
	So it will suffice to bound the length of $W$ in terms of $p$ and of lengths in $H$ of subwords of $\tilde{u}$ and $\varphi^{-r}(\tilde{v})$.	
		
	Our $u$, $v$ and $w$ form either  the I- or $\H$-configuration of Figure~\ref{fig:fbc determine X}.
	
	In the case of the I-configuration, $w = u_0v_0^{-1}s^r$ where  $u_0$ and $v_0$ are prefixes of $\tilde{u}$  and $\varphi^{-r}(\tilde{v})$.  	Then  $\abs{w}_H \leq   \abs{u_0}_H + \abs{v_0}_H +r$, and so $W =w$ will be a conjugator which, by the discussion above, satisfies the required length bound.
	
	Now consider the case of the $\H$-configuration.	  Proposition~\ref{conjugator-like problem in F} tells us that we have a conjugator $\tilde{w} s^r$, where $\tilde{w} = u_0 X v_0^{-1}$, with $u_0$ and $v_0$ prefixes of $\tilde{u}$ and $\varphi^{-v}(\tilde{v})$ respectively, and the form of $X$ following one of \eqref{form of x 0}--\eqref{form of x 3}.
	It suffices for us to show  $\ell(X)$ is at most a constant times  $\abs{u}_H + \abs{v}_H$.
 Lemma~\ref{lem:chunk lenght bound} would give this bound  but for $M_1$ in case  \eqref{form of x 2} and $\hat{x}$ in  case  \eqref{form of x 3}.
	
	As remedy, in the event of \eqref{form of x 2},  we use the conjugator from Lemma~\ref{lem:swap for a linear conjugator}.
	As $L$, $S$, $M_2$, $P$, and $R$ are bounded as required, and $qp$ is bounded by \eqref{eq:bound on qp}, the required bound on $\abs{w}_H$ follows.
	
	 In the event of case \eqref{form of x 3},  we iterate this process.
	By Proposition~\ref{conjugator-like problem in F} we know that   $\rank(\hat{x}) < \rank(x)$ and  $\hat{x}$ occurs in a solution to  an $\H$-twisted conjugacy problem, namely 
	\begin{equation} \label{eq:final proof iterate}
	\varphi^{-p}(\hat{u}_0 \hat{x}\hat{v}_0^{-1})  \ = \  \hat{u}_1^{-1} \hat{x} \hat{v}_1
	\end{equation}
	where  $\hat{u}_0$, $\hat{u}_1$, $\hat{v}_0$, and $\hat{v}_1$ are words whose   lengths in $H$ are at most a constant multiple of   $\abs{u}_H + \abs{v}_H$  by Proposition~\ref{prop:fbc subword length}.   
	Lemma~\ref{lem:annoying iteration} shows how an $\hat{X}$ solving this new  $\H$-twisted conjugacy problem leads to an $X$ solving the earlier one, and that if  $\hat{X}$ has length at most a constant times $\abs{u}_H + \abs{v}_H$, then the same will be true of $X$.
	We  reapply Proposition~\ref{conjugator-like problem in F}, and again, if we hit case~\eqref{form of x 0} or \eqref{form of x 1} then we can stop. In case~\eqref{form of x 2}, a short conjugator is found via Lemma~\ref{lem:swap for a linear conjugator}.
	If we hit case~\eqref{form of x 3} then we iterate down to a lower rank again.
	
The maximum number of times we can iterate through case~\eqref{form of x 3}  is $m-1$ times.  This will bring us to rank $1$ (if the process has not yet terminated) and then case \eqref{form of x 0} will apply.   So this process will terminate at an $X$ and so a $W$ of suitably bounded length.

\section{The algorithm: completing our proof of Theorem~\ref{CP of free-by-cyclic}} \label{the alg}

Here, in outline, is our  algorithm for the conjugacy and conjugacy search problems for $H$.
 
\textbf{Input:} Words $u$ and $v$ on $a_1^{\pm 1} , \ldots , a_m^{\pm 1} , s^{\pm 1}$.

\step{1}
Convert $u$ and $v$ to normal forms $\tilde{u}s^p$ and $\tilde{v}s^q$ respectively.
If $p\ne q$, then stop and declare $u$ is not conjugate to $v$.  If $p = q <0$, then replace $u$ and $v$ by their inverses and return to the start.
\nopagebreak

Time required: polynomial in $\ell(u)+\ell(v)$.

\step{2}
If $p=q=0$, then run the algorithm of Proposition~\ref{0 twisted conj problem} \eqref{case:0-twisted compl} solving the 0-twisted conjugacy problem.
If it declares the 0-twisted conjugacy problem has no  solution, then declare $u$ is not conjugate to $v$.
Otherwise it outputs a solution $(r,\tilde{w})$, so stop and declare $\tilde{w}s^r$ is a conjugator.
\nopagebreak

Time required: polynomial in $\ell(\tilde{u}) + \ell(\tilde{v})$.

\step{3} We have $p=q>0$.
Let $\cal{I}$ be the set of all pairs $(\tilde{w},r)$, where $0\le r  < p$, and $\tilde{w}$ is a word of the form $UV$ where $U$ is a prefix of $\tilde{u}$ and $V^{-1}$ is a prefix of $\varphi^{-r}(\tilde{v})$.
For each $(\tilde{w},r)$ in $\cal{I}$, check whether $\tilde{u} \varphi^{-p}(\tilde{w}) = \tilde{w} \varphi^{-r}(\tilde{v})$  (as per the I-twisted conjugacy problem). If a solution $(\tilde{w},r)$ is found, then the algorithms declares that $u$ and $v$ are conjugate and outputs $\tilde{w}s^r$ as a conjugator.
If no solution is found we continue to the next step.
\nopagebreak

Time required: the number of entries
on the list $\mathcal{I}$ is bounded by a polynomial in $p + \ell(\tilde{u}) + \ell(\tilde{v})$ and the obvious solution to the
word problem in $F$ runs in linear time, so overall this step runs in  time polynomial in $\abs{p} + \ell(\tilde{v}) + \ell(\tilde{v})$.
	
\step{4}
Run the algorithm of Corollary~\ref{conjugator-like problem in F solution}  for the $\H$-twisted conjugacy problem.
If it declares there is no solution, stop and declare that $u$ and $v$ are not conjugate.
If it exhibits a solution $(r,x,u_0,v_0,u_1,v_1)$, then declares that $u$ and $v$ are conjugate and output $u_0 x v_0^{-1} s^r$ as a conjugator.
\nopagebreak

Time required: polynomial in $\abs{p} + \abs{\tilde{u}}_H + \abs{\tilde{v}}_H$.

\medskip

\textbf{Overall time required:}
Since $\abs{\tilde{u}}_H \leq \ell(\tilde{u}) \leq C \ell(u)^m$ and $\abs{\tilde{v}}_H \le \ell(\tilde{v}) \leq C \ell(v)^m$ for a suitable constant $C>0$, and $\abs{p} \leq \ell(u)$,  the total running time of the algorithm is polynomial in $\ell(u) + \ell(v)$.

\begin{remark} \label{upgrade to linear remark}  
The algorithm described above can be modified as follows to output in polynomial time a conjugator  $W$ (if one exists) with $\ell(W)$ at most a constant times $\ell(u) + \ell(v)$. We describe the required changes.  First,
when the algorithm  of Lemma~\ref{inductive alg} finds a conjugator of form \eqref{form of x 2} its output  includes the subword $M_1$ or $M_1^{-1}$.  Add  an extra step that  replaces this $M_1$ by $s^{pq}$.  (Lemma~\ref{lem:swap for a linear conjugator} confirms that the result remains a conjugator.)  This will produce  a word  $W_0$ which is a conjugator whose \emph{length in $H$} is bounded by a linear 
function of  $\abs{u}_H + \abs{v}_H\le\ell(u) + \ell(v)$.  This means that   there is a word $W$ on    $a_1^{\pm 1}, \ldots, a_m^{\pm 1}, s^{\pm 1}$ that equals $W_0$ in $H$  and has length $\ell(W)\le \ell(u) + \ell(v)$.  It remains to argue that we can further adapt the algorithm to exhibit such a $W$.   The word $W_0$ is assembled from words derived from   subwords of $\tilde{u}$ and $\tilde{v}$  as described in Section~\ref{sec:CL of H}.     We used  Proposition~\ref{prop:fbc subword length} to  bound the lengths (in $H$) of such subwords in terms of $|u|_H$ or $|v|_H$.  Our proof of   Proposition~\ref{prop:fbc subword length} is constructive. In particular it can be adapted to a polynomial time algorithm that, for example, takes a subword $u'$ of $\tilde{u}$, where $u=\tilde{u}s^p$ in normal form, and gives a word on $a_1^{\pm 1}, \ldots, a_m^{\pm 1}, s^{\pm 1}$  that equals $u'$ in $H$ and whose length is at most a constant times    $|u|_H$.  We can then assemble $W$  from words obtained in  this manner.  
\end{remark}

\section{{An alternative approach}}\label{s:last}

In this section we outline an alternative proof of  
Theorem~\ref{CP of free-by-cyclic} that 
is based on the structure of $H_m$ as an iterated HNN extension. This
alternative approach will be developed in detail in \cite{BrRiSa3} and
applied to a wider class of free-by-cyclic groups.

We regard
$H=H_m$ as an $(m-1)$-fold
iterated HNN extension of $H_1=\<s,a_1\>\cong \Z^2$ where at the 
$j$-th stage the base group is $H_j:=\<s,a_1,\dots, a_{j}\>$, the stable
letter is $a_{j+1}$, the associated (cyclic) subgroups are $\<s\>$ and 
$\<sa_{j}^{-1}\>$, and the relation $a_{j+1}^{-1}s a_{j+1} = sa_{j}^{-1}$ holds.
This point of view enables one to argue by induction on $m$ and
appeal to the technology of corridors
to analyse van Kampen diagrams and their annular
analogues over the natural presentations of these groups.
But in keeping with the viewpoint of this article, we shall  
suppress the use of diagrams here and concentrate on the algebraic translation
of the insights that they provide.

There is a classical approach to the conjugacy problem in HNN extensions
based on Collins' Lemma \cite{CollinsConj}---see \cite{LS}, page 185, for example. 
This simplifies in the case where the associated subgroups are cyclic,
as we shall now explain.

Let $G= (G_0, t \mid t^{-1}\a t=\b)$ be an HNN extension
where $A=\<\a\>$ and $B=\<\b\>$ are 
infinite cyclic.  
We fix a generating set $S$ for $G_0$ that includes 
$\a$ and $\b$. A word $U$ in the alphabet $S^{\pm 1}$ is  
in {\em cyclically reduced HNN form} if $U^2$ does not
contain a pinch---i.e., a subword $t^{-1}ct$ with $c\in A$ or 
$t d t^{-1}$ with $d\in B$. 

For simplicity, we assume that 
distinct powers of $\b$ are not conjugate
in $G$ and that no power of $\a$ is conjugate to a power of $\b$ in $G_0$.
A straightforward analysis of annular diagrams yields the following version
of Collins' Lemma in this simplified setting.

\begin{lemma}[Collins' Lemma]
Assume that $U$ and $V$ are words in 
cyclically reduced HNN form. If $U$ is conjugate to $V$ in $G$, then either
\begin{enumerate}
\item[i]  $U$ and $V$ contain no occurrences of $t^{\pm 1}$ and
either they are conjugate in $G_0$ or else
one is conjugate into $A$ and the other is conjugate into $B$; or else 
\item[ii]   both $U$ and $V$ contain an occurrence of $t^{\pm 1}$ and
there are cyclic permutations $U'$ of $U$ and $V'$ of $V$
and an integer $q$ such that $\a^{-q} U' \a^q = V'$ in $G$.
\end{enumerate}
\end{lemma} 

\subsection{The algorithm for Theorem~\ref{CP of free-by-cyclic}.} 
We regard $H_m$ as an HNN extension of $H_{m-1}$
as described in the second paragraph. 
In the language used above, $G=H_m$ while
$G_0=H_{m-1},\, A=\<s\>,\ B=\<\b\>$ and $t=a_k$,
where $\beta$ is a generator we have added with $\beta=sa_{m-1}^{-1}$ in $H_m$. 

Proceeding by induction,  we may
assume that we have a polynomial time algorithm 
to decide conjugacy in $H_{m-1}$. 
Given two words $u, v$ in the generators of $H_m$ (with $\beta$ included)
we rewrite them into cyclically HNN reduced words $U, V$. This is achieved
by first transforming $u$ and $v$ to 
reduced HNN form by removing pinches and then examining 
cyclic
permutations of $u$ and $v$,
removing any additional pinches that appear. The second step may need to be repeated
several times, but the word is shortened each time. Both steps can be
done in polynomial time without increasing the length of the words.  

We are now able to apply  Collins' Lemma.  
If there are no occurrences of $a_m^{\pm 1}$ in $U'$ and $V'$, 
then we are in case (i) and we apply the algorithm for $H_{m-1}$.
Otherwise we are in case (ii) and we are left to determine if there is an integer $p$
such that
$
s^{-q} U' s^q = V'.
$
(Recall that $A=\<\a\>= \<s\>$.)

In polynomial time, we can rewrite $U'$ and $V'$ into normal
form $\tilde{U}'s^r$ and $\tilde{V}'s^{r'}$,
where the lengths of $\tilde{U}'$ and $\tilde{V}'$ are bounded 
polynomially by $|U|$ and $|V|$. If $r\neq r'$, then
we stop and declare that $U$ is not conjugate to $V$. If $r=r'$,
then we are reduced to deciding if there is a positive integer
$p$ such that $\phi^p(\tilde{U}')=\tilde{V}'$ or $\phi^p(\tilde{V}')=\tilde{U}'$. 
The range of possible $p$ is bounded by a linear function of $|U|+|V|$,
by considerations of growth, as in Section~\ref{distortion section}. And for 
each specific $p$, we can evaluate $\phi^p(U)$ naively  (letter by letter)
and freely reduce to see if it is equal to $V$.
As $\varphi$ has polynomial
growth, these evaluations can be done in polynomial time.

This  algorithm, as we have described it, does not provide the linear upper bound on 
conjugator length that is required for Theorem~\ref{CL of free-by-cyclic}.
The main argument in \cite{BrRiSa3} overcomes this limitation
with an alternative endgame that  makes greater use of the structure of $H_m$ as an iterated HNN extension.

\bibliographystyle{alpha}
\bibliography{bibli}

\ni {Martin R.\ Bridson} \\
Mathematical Institute, Andrew Wiles Building, Oxford OX2 6GG, United Kingdom \\ {bridson@maths.ox.ac.uk}, \
\href{http://www2.maths.ox.ac.uk/~bridson/}{https://people.maths.ox.ac.uk/bridson/}

\ni  {Timothy R.\ Riley} \rule{0mm}{6mm} \\
Department of Mathematics, 310 Malott Hall,  Cornell University, Ithaca, NY 14853, USA \\ {tim.riley@math.cornell.edu}, \
\href{http://www.math.cornell.edu/~riley/}{http://www.math.cornell.edu/$\sim$riley/}

\ni {Andrew W.\ Sale} \rule{0mm}{6mm} \\
Department of Mathematics,
University of Hawaii at Manoa,
Honolulu, HI 96822, USA \\   \href{http://math.hawaii.edu/~andrew/}{http://math.hawaii.edu/$\sim$andrew/}

\end{document}